\documentclass{amsart}

\usepackage{amssymb,amsxtra,latexsym,enumerate}

\usepackage[osf]{newpxtext} \usepackage{textcomp}
\usepackage[cmintegrals]{newpxmath} \usepackage[scr=rsfs]{mathalfa}

\usepackage[varqu,varl]{inconsolata}

\usepackage{tikz-cd} \tikzcdset{arrow style=math font}
\tikzcdset{diagrams={nodes={inner sep=3pt}}}

\usepackage{hyperref,xcolor}

\definecolor{ceruleanblue}{rgb}{0.16, 0.32, 0.75}
\hypersetup{colorlinks=true,allcolors=ceruleanblue}

\makeatletter \let\@wraptoccontribs\wraptoccontribs \makeatother
\newcommand\note[1]%
                {$^\dagger$\marginpar{\footnotesize{$^\dagger${#1}}}}

\numberwithin{equation}{section}

\swapnumbers

\newtheorem{theorem}[equation]{Theorem}
\newtheorem{lemma}[equation]{Lemma}
\newtheorem{proposition}[equation]{Proposition}
\newtheorem{corollary}[equation]{Corollary}

\theoremstyle{definition}
\newtheorem{definition}[equation]{Definition}
\newtheorem{example}[equation]{Example}

\newtheorem{remark}[equation]{Remark}

 \newcommand\lie{\mathfrak}

\newcommand{\g}{\lie{g}}

\newcommand{\ttt}{\lie{t}}

\newcommand\bb{\mathbb}

 \newcommand\Z{\bb{Z}} \newcommand\Q{\bb{Q}}
\newcommand\R{\bb{R}} \newcommand\C{\bb{C}} 
\newcommand\T{\bb{T}}

\newcommand\ca{\mathscr}

\DeclareMathOperator\ann{ann}

\DeclareMathOperator\Hom{Hom}

\DeclareMathOperator\id{id}

\DeclareMathOperator\Lie{Lie}

\DeclareMathOperator\pr{pr}

\DeclareMathOperator\stab{stab}

\DeclareMathOperator\Vect{Vect}

\usepackage{manfnt}

\newcommand\group[1]{{\text{\bf#1}}}

\newcommand\B{\group{B}}

\newcommand\G{\group{G}}

\newcommand\qu[1][\kern.3ex]{/\kern-.7ex/_{\kern-.4ex#1}}
\newcommand\bigqu[1][\,\,]{\big/\kern-.85ex\big/_{\!\!#1}}

\newcommand\powl{[\kern-.3ex[} \newcommand\powr{]\kern-.3ex]}
\newcommand\bigpowl{\bigl[\kern-.6ex\bigl[}
    \newcommand\bigpowr{\bigr]\kern-.6ex\bigr]}

\newcommand\sur{\mathrel{\to\kern-1.8ex\to}}

\newcommand\longsur{\mathrel{\longrightarrow\kern-1.8ex\to}}

\newcommand{\n}{^{-1}}

\newcommand\bu{\bullet}

\newcommand\coarse{{\mathrm{coarse}}}

\newcommand\stack{\group}

\newcommand\X{\stack{X}} 
\newcommand\Y{\stack{Y}}

\newcommand\stackmorphism{\boldsymbol}

\newcommand\bmu{{\stackmorphism{\mu}}}

\newcommand\bphi{{\stackmorphism{\phi}}}

\newcommand\bomega{{\stackmorphism{\omega}}}

\newcommand\F{{\ca{F}}}

\begin{document}


\title{Toric Symplectic Stacks}
\author{Benjamin Hoffman}
\maketitle

\begin{abstract}
We give an intrinsic definition of compact toric symplectic stacks, and show that they are classified by simple convex polytopes equipped with some additional combinatorial data. This generalizes Delzant's classification of compact toric symplectic manifolds. As an application, we show that any compact toric symplectic stack can be deformed to an ineffective toric orbifold.
\end{abstract}

\section{Introduction}

A compact connected $2n$-dimensional symplectic manifold $(M,\omega)$ is \emph{toric} if there is an effective Hamiltonian action on $M$ by an $n$-dimensional torus $\T$. The image of $M$ under the moment map $\mu$ is a convex polytope in $\mathfrak{t}^*$, the linear dual of the Lie algebra $\ttt$ of $\T$. The moment polytope has these three propertes:
\begin{itemize}
\item It is \emph{simple}, meaning there are $n$ facets meeting at every vertex.
\item It is \emph{rational}, meaning that normal vectors to the facets of $\mu(V)$ can be chosen in the cocharacter lattice $\Hom(S^1,\T)\cong \Z^n$ in $\ttt\cong \R^n$.
\item It is \emph{smooth}, meaning that at each vertex $v$ of $\mu(M)$, the primitive normal vectors to the facets meeting at $v$ form a $\Z$-basis of $\Hom(S^1,\T)$.
\end{itemize}
In \cite{delzant;hamiltoniens-periodiques}, Delzant showed that the assignment taking $(M,\omega,\mu,\T)$ to its moment polytope $\mu(M)\subset \ttt^*$ gives a one-to-one correspondence between the set of $2n$-dimensional toric symplecic manifolds (up to isomorphism), and the set of $n$-dimensional polytopes which are simple, rational, and smooth (up to the action of $\R^n\rtimes Gl(n,\Z)$ on $\ttt^*$).

In the proof of his theorem, Delzant showed how to reconstruct a compact toric symplectic manifold from its moment polytope. If one follows this construction, beginning with a polytope which is simple and rational, but not smooth, then one recovers not a manifold, but an orbifold. Toric symplectic orbifolds were defined and studied in \cite{lerman-tolman;hamiltonian-torus-actions-symplectic-orbifolds}. The authors proved a generalization of Delzant's theorem: compact toric symplectic orbifolds are classified by simple rational polytopes, with facets labeled by positive integers.

A natural question is what happens when one begins Delzant's construction with a polytope which is simple, but neither rational nor smooth. In this case, one recovers not an orbifold but an \'etale differentiable stack, equipped with an action of a stacky torus (see Section \ref{section3}). The two goals of the present paper are to give an intrinsic definition of a compact toric symplectic stack, and to prove Thereom \ref{maintheorem}, which is a generalization of Delzant's theorem to this context. Theorem \ref{maintheorem} states that compact toric symplectic stacks are classified by simple polytopes, decorated with some additional combinatorial data.

Compact toric symplectic manifolds are beloved examples in the study of Hamiltonian group actions, partly because much of the symplectic geometry of a compact toric symplectic manifold $M$ can be translated into the combinatorics of its moment polytope. For instance, the $\T$-equivariant cohomology of $M$ and the Liouville volume of $(M,\omega)$ can be read from the polytope $\mu(M)$. One possible application of the present work is a description of the geometry of compact toric symplectic stacks in terms of their moment polytopes.

This connection between non-rational polytopes and symplectic geometry has been investigated before, for instance in \cite{katzarkovDefinitionNoncommutativeToric2014, lin-sjamaar;presymplectic, prato;non-rational-symplectic, ratiu-zung;presymplectic}. Our approach differs from these in several important respects. First, we give an intrinsic (atlas-independent) definition of compact toric symplectic stacks. In \cite{ratiu-zung;presymplectic} for instance, the authors assume the existence of a presymplectic atlas with certain restrictions on the null foliation of the presymplectic form. Because we begin with an intrinsic definition of compact toric symplectic stacks, our approach to proving our classification Theorem \ref{maintheorem} is fundamentally different from the proof of the related Theorem 3.6 in \cite{ratiu-zung;presymplectic}. Similarly, our Theorem \ref{convexitythm} says that the image of a compact toric symplectic stack under the moment map is a simple convex polytope. Convexity is addressed by \cite{lin-sjamaar;presymplectic} and \cite{ratiu-zung;presymplectic}; however, these approaches require that the stack in question has an atlas of a particular form. Our proof of Theorem \ref{convexitythm} works with the intrinsic definition of a compact toric symplectic stack in full generality.

Second, our approach differs from others in that we allow for ineffective stacks, which are stacks with a non-trivial global isotropy group. In particular, the notion of a quasilattice was first introduced in \cite{prato;non-rational-symplectic}, but there it was assumed that the structure map $\partial\colon Q\to \R^n$ is injective (in the notation of Sections \ref{section2} and \ref{section3}).

Finally, we extend Delzant's construction to allow for one-parameter families of polytopes and compact toric symplectic stacks. As a consequence, we find that any compact toric symplectic stack can be deformed to an ineffective compact toric symplectic orbifold. An example of a one-parameter family of toric quasifolds (stacks) was given by Battaglia, Prato, and Zaffran in \cite{battaglia-prato-zaffran;one-parameter}.

The theory of toric algebraic stacks has been developed by several sources, and in the most generality by Geraschenko and Satriano in \cite{geraschenko-satriano;toric} and \cite{geraschenko-satriano;toric2}. It is illuminating to compare their toric algebraic stacks (over $\C$) to our compact toric symplectic stacks. Toric algebraic stacks have an action of an algebraic group stack, while compact toric symplectic stacks have an action of a Lie group stack which is often the quotient of a compact torus by a dense immersed subgroup. In particular, by Definition 2.16 of \cite{geraschenko-satriano;toric}, any toric algebraic stack has an open dense substack which is equivalent (as an algebraic stack) to an algebraic group stack. On the other hand, compact toric symplectic stacks have an open dense substack which is equivalent (as a differentiable stack) to the stacky quotient of $(\C^\times)^n$ by a subgroup which can be dense, immersed, and non-algebraic. There are then many examples of compact toric symplectic stacks which do not arise in the algebraic context.

The paper is organized as follows. In Section \ref{section2}, we recall some background on Hamiltonian actions on symplectic stacks, which was developed in detail in \cite{hoffman-sjamaar;hamiltonian-stack}. In Section \ref{section3} we define stacky tori and compact toric symplectic stacks. In Section \ref{section4} we show that compact toric symplectic stacks are locally presented by certain action groupoids. In Section \ref{section5} we define the stacky moment polytope of a compact toric symplectic stack. In Section \ref{section6} we state the main theorem of this article, which is then proved in Sections \ref{section7} through \ref{section9}. In particular, Sections \ref{section7} and \ref{section8} deal with the uniqueness part of Theorem \ref{maintheorem}, while Section \ref{section9} concerns the existence part, as well as deformations of polytopes.

\textbf{Acknowledgements.} We are grateful to R. Sjamaar for suggesting this project, and to  F. De Keersmaeker, N. Hemelsoet, J. Lane, I. Pendleton, R. Sjamaar, and X. Tang for their useful comments and discussions throughout this project. The author was supported in part by the National Science Foundation Graduate Research Fellowship under Grant Number DGE-1650441.

\section{Background on Hamiltonian actions on symplectic stacks}

\label{section2}

In this section we very briefly recall some background material on Hamiltonian actions on stacks. This was developed in detail in \cite{hoffman-sjamaar;hamiltonian-stack}, and the reader is referred there for more information. Some of the material is motivated by~\cite{lerman-malkin;deligne-mumford}. For background on differentiable stacks in general we refer the reader to \cite{bursztyn-noseda-zhu;principal-stacky-groupoids, metzler;topological-smooth-stack, noohi;foundations-topological-stacks}.

\subsection*{Lie groupoids}

Lie groupoids $X_\bu$ have the source, target, identity, inverse, and multiplication structure maps denoted by $s,t,u,(\cdot)\n,$ and $\circ$, respectively. We assume that $X_0$ and $X_1$ are both Hausdorff and second-countable. If $G$ acts on $X$, then we write $G\ltimes X\rightrightarrows X$ for the action groupoid. Sometimes we will denote a groupoid by its manifold of arrows, so for instance $G\ltimes X\rightrightarrows X$ may be written $G\ltimes X$. A smooth manifold $M$ will be identified with the trivial groupoid $M\rightrightarrows M$. If $X_\bu$ is a Lie groupoid and $x\in X_0$, then the isotropy group of $x$ is the group $s\n(x)\cap t\n(x)$ of arrows which start and end at $x$.

Let $\mathbf{LieGpd}$ denote the $2$-category of Lie groupoids; see~\cite{moerdijk-mrcun;foliations-groupoids} for a detailed description of its morphisms and 2-morphisms.

Unless otherwise noted, the symbol $\ltimes$ will be used exclusively for action groupoids, and not for semidirect products of groups.

\subsection*{Morita equivalence}

Let $X$ be a manifold and $Y_\bu$ be a Lie groupoid, and let $\phi\colon X \to Y_0$ be a smooth map. Then $\phi$ is \emph{essentially surjective} if $t\circ\pr_1\colon Y_1\times_{s,Y_0,\phi}X\to Y_0$, which sends
$(g,x)\mapsto t(g)$ is a surjective submersion. More generally, a morphism $\phi_\bu \colon X_\bu\to Y_\bu$ of Lie groupoids is called essentially surjective if the map $\phi_0$ is essentially surjective. A morphism $\phi_\bu\colon X_\bu \to Y_\bu$ is
\emph{fully faithful} if the square
\[
\begin{tikzcd}[row sep=large]
X_1\ar[r,"\phi_1"]\ar[d,"{(s,t)}"']&Y_1\ar[d,"{(s,t)}"]\\
X_0\times X_0\ar[r,"\phi_0\times\phi_0"]&Y_0\times Y_0
\end{tikzcd}
\]
is a fibered product of manifolds.  If $\phi_\bu$ is both essentially
surjective and fully faithful then $\phi_\bu$ is a \emph{Morita
  morphism}.  If there is a zigzag of Morita morphisms $X_\bu\to
Y_\bu$ and $X_\bu\to Z_\bu$, then $Y_\bu$ and $Z_\bu$ are \emph{Morita
  equivalent}. Morita equivalence is an equivalence relation on Lie groupoids, and Morita equivalent groupoids will present equivalent differentiable stacks (see Theorem~\ref{presentationtheorem} below).
    
Let $X$ be a manifold, $Y_\bu$ a Lie groupoid, and $\phi\colon X\to Y_0$ an essentially surjective map, then define the \emph{pullback groupoid} $\phi^*Y_\bu:=(\phi^* Y_1\rightrightarrows \phi^*Y_0 = X)$. The space of arrows of $\phi^* Y_\bu$ is 
 \begin{align*}
\phi^*Y_1
&:=X\times_{\phi,Y_0,s} Y_1\times_{t,Y_0,\phi}X \\
&\hphantom{:}=\{\,(x,g,y)\in X\times Y_1\times
X\mid\phi(x)=s(g)\text{ and }t(g)=\phi(y)\,\}.
\end{align*}
Then $\phi^* Y_\bu$ is Morita equivalent to $Y_\bu$. 

\subsection*{Differentiable stacks}

We consider the 2-category $\mathbf{Stack}$ stacks over the category $\mathbf{Diff}$ of smooth manifolds, with the open cover Grothendieck topology; see~\cite[Section~2]{metzler;topological-smooth-stack} for further introduction to these notions in general. Equivalence of stacks will be denoted by the symbol $\simeq$. We denote stacks and their morphisms by bold letters, eg $\X$ and $\bmu\colon \X \to \R$. 

 Denote by $\B$ the 2-functor which takes a Lie groupoid $X_\bu$ to its associated stack $\B X_\bu$; see~\cite[Section~3.1]{metzler;topological-smooth-stack} and \cite[Definition~3.3,~Section~12]{noohi;foundations-topological-stacks} for a detailed description of this functor. If $\X$ is a stack over $\mathbf{Diff}$, then $\X$ is a \emph{differentiable stack} if there exists a Lie groupoid $X_\bu$, and an equivalence $\B X_\bu\simeq \X$. In this case we say $X_\bu$ \emph{presents} $\X$, and the equivalence $\B X_\bu \simeq \X$ is a \emph{presentation of $\X$}. By the 2-Yoneda lemma, if $M$ is a smooth manifold we identify $M=\B M$. Let $\mathbf{DiffStack}$ denote the full sub-2-category of $\mathbf{Stack}$, whose objects are differentiable stacks. Denote by $\star$ a terminal object in $\mathbf{Stack}$.
 
The following makes more precise the relationship between Lie groupoids and differentiable stacks. It is a restatement of the discussion following~\cite[Proposition~60]{metzler;topological-smooth-stack}. (See also \cite[Theorem~2.2]{behrend-xu;stacks-gerbes}).
 
 \begin{theorem}\label{presentationtheorem}  Let $\phi_\bu\colon X_\bu \to Y_\bu$ be a map of Lie groupoids. Then $\B\phi_\bu \colon \B X_\bu \to \B Y_\bu$ is an equivalence of stacks if and only if $\phi_\bu$ is a Morita morphism. Any map $\bphi\colon \B X_\bu \to \B Y_\bu$ of differentiable stacks is $2$-isomorphic to a map of the form 
 \[
 \B X_\bu \xrightarrow{(\B \psi_\bu)\n} \B X_\bu' \xrightarrow{\B \psi'_\bu} \B Y_\bu,
 \]
 where $\psi_\bu\colon X_\bu'\to X_\bu$ is a Morita morphism, and $(\B \psi_\bu)\n$ is any inverse equivalence to $\B \psi_\bu$. Consequently, $\B X_\bu$ is equivalent to $\B Y_\bu$ if and only if $X_\bu$ is Morita equivalent to $Y_\bu$.
 \end{theorem}
 
\begin{remark} Theorem~\ref{presentationtheorem} can be strengthened using the language of bicategories, following~\cite{pronkEtenduesStacksBicategories1996}. Let $\mathbf{LieGpd}[M\n]$ denote the bicategory formed by formally inverting the class $M$ of Morita morphisms in $\mathbf{LieGpd}$. Then $\B\colon \mathbf{LieGpd} \to \mathbf{DiffStack}$ extends to a map of bicategories
 $\B^{loc}\colon \mathbf{LieGpd}[M\n] \to \mathbf{DiffSt}$, which is an equivalence of bicategories. \end{remark}
 
An \emph{atlas} of a stack $\X$ is a representable surjective submersion from a manifold $X\to \X$. Then $\X$ is presented by the Lie groupoid $X\times_\X X \rightrightarrows X$. Conversely, applying $\B$ to the natural inclusion of Lie groupoids $X_0\to X_\bu$ gives an atlas $X_0\to \B X_\bu$ of the stack $\B X_\bu$.

To a differentiable stack $\X\simeq \B X_\bu$ we associate a topological space $\X^{coarse}$, the \emph{coarse quotient}. The coarse quotient is homeomorphic to $X_0/X_1$, which is the topological space formed by identifying points of $X_0$ which lie in the same $X_1$-orbit. This space typically fails to be Hausdorff. Similarly, to a morphism of differentiable stacks $\stackmorphism{\phi}\colon \X\to \Y$ we associate the morphism $\stackmorphism{\phi}^{coarse} \colon \X^{coarse} \to \Y^{coarse}$ of coarse quotient spaces. The passage from differentiable stacks to their coarse quotients is functorial~\cite[Proposition~4.15]{noohi;foundations-topological-stacks}.

Let $f\colon \X \to \R^n$ be a function on a differentiable stack $\X$. A point $p \in \R^n$ is a \emph{regular value of $f$} if, for every atlas $X\to \X$, the point $p$ is a regular value of the smooth map $X\to \X \xrightarrow{f} \R^n$. A point $x \in \X^{coarse}$ is a \emph{regular point of $f^{coarse}\colon \X^{coarse} \to \R^n$} if for every groupoid presentation $X_1\rightrightarrows X_0\to \X$, every point $y\in X_0$   which is sent to $x$ under $X_0 \to X_0/X_1\cong \X^{coarse}$ is a regular point of $X_0\to \X \xrightarrow{f} \R^n$.

It is straightforward to show the following conditions are equivalent to the definitions in the previous paragraph: A point $p\in \R^n$ is a regular value of $f$ if there exists an atlas $X\to \X$ such that $p$ is a regular value of $X\to \X \to \R^n$. And,
a point $x\in \X^{coarse}$ is a regular point of $f^{coarse}$ if there exists a groupoid presentation $X_1\rightrightarrows X_0 \to \X$ and a point $y\in X_0$, such that $y$ is sent to $x$ under $X_0\to X_0/X_1\to \X^{coarse}$ and such that $y$ is a regular point of $X_0\to \X\xrightarrow{f}\R^n$.

\subsection*{Foliation groupoids and \'etale stacks}

A Lie groupoid $X_\bu$ is \emph{\'etale} if the source map $s$ is a local diffeomorphism. More generally, $X_\bu$ is called a \emph{foliation groupoid} if it is Morita equivalent to an \'etale Lie groupoid, or equivalently if it has discrete isotropy groups (see \cite[Theorem~1]{crainic-moerdijk;foliation-cyclic}). The name is due to the fact that the connected components of $X_1$ orbits form a foliation of $X_0$. If $\F$ is the foliation of $X_0$ by $X_1$ orbits, then $T\F$ is a constant rank subbundle of $TX_0$. 

A differentiable stack $\X$ is \emph{\'etale} if it is equivalent to $\B X_\bu$, where $X_\bu$ is an \'etale groupoid. Then any groupoid presentation of $\X$ is a foliation groupoid. We do not require that $\X^{coarse}$ is Hausdorff.

A Lie groupoid $X_\bu$ is \emph{proper} if the map $(s,t)\colon X_1\to X_0\times X_0$ is a proper map. A differentiable stack $\X$ is proper if it admits a presentation by a proper Lie groupoid. 

If $X_\bu$ is an \'etale groupoid, and $g\in X_1$, then $g$ induces a germ $\gamma(g)$ of a diffeomorphism of $X_0$ which takes $s(g)$ to $t(g)$. We say $X_\bu$ is \emph{effective} if $\gamma(g) = \gamma(h)$ implies $g=h$ for all $g,h\in X_1$. A differentiable stack $\X$ is \emph{effective} if it admits a presentation by an effective Lie groupoid $X_\bu$.

A stack $\X$ is an \emph{orbifold} if it is both proper and \'etale. If $\X$ is an effective stack and an orbifold, then $\X$ is an \emph{effective orbifold}. See~\cite{lerman;orbifolds-as-stacks, moerdijkOrbifoldsSheavesGroupoids1997} for justification of this terminology.

\subsection*{Basic vector fields and forms}

Let $X_\bu$ be a foliation groupoid. For $k\ge 0$, consider the vector spaces of basic $k$-forms
\begin{align*}
\Omega^k_0(X_\bu) & =\{\omega_{X_\bu}:=(\omega_{X_0},\omega_{X_1}) \in \Omega^k(X_0)\times \Omega^k(X_1) \mid s^*\omega_{X_0}= \omega_{X_1}=t^*\omega_{X_0} \} \\
& \cong\{\omega \in \Omega^k(X_0) \mid s^* \omega = t^*\omega\}
\end{align*}
The exterior derivative restricts to a map $d  \colon  \Omega^k_0(X_\bu)  \to \Omega^{k+1}_0(X_\bu)$ sending $(\omega_{X_0},\omega_{X_1})  \mapsto (d\omega_{X_0},d\omega_{X_1})$. 

Let $\F$ be the foliation of $X_0$ by $X_1$ orbits, and let $N_0=TX_0/T\F$ be the normal bundle to $T\F$. Let $N_1 = TX_1/(\ker(ds)+\ker(dt))$. As a consequence of~\cite[Lemma~5.3.2]{hoffman-sjamaar;hamiltonian-stack}, $N_1$ is a constant rank vector bundle over $X_1$, and the source and target maps induce isomorphisms $s^*N_0\cong N_1 \cong t^* N_0$ of vector bundles over $X_1$.

 The space of \emph{basic $k$-vector fields} on $X_\bu$ is
 \begin{align*}
\Vect_0^k(X_\bu) & = \{v_{X_\bu}:=(v_{X_0},v_{X_1})\in \Gamma(\wedge^k N_0)\times \Gamma(\wedge^k N_1)\mid s^*v_{X_0} = v_{X_1} = t^*v_{X_0}\} \\
& \cong \{v \in \Gamma(\wedge^k N_0) \mid s^*v =t^* v\}.
\end{align*}
The Lie bracket on $\Vect(X_0)\times \Vect(X_1)$ descends to a Lie bracket on 
\[
\Vect_0(X_\bu):=\Vect_0^1(X_\bu).
\]

\begin{remark}
It is typical~\cite{berwick-evans-lerman;lie-2-algebras-vector, hepworth;vector-flow-stack} to take use the Lie 2-algebra \emph{multiplicative vector fields} on $X_\bu$ as a model for vector fields on the stack $\B X_\bu$; recall that a multiplicative vector field is a Lie groupoid morphism $v_\bu \colon X_\bu \to T X_\bu$ satisfying $p_\bu \circ v_\bu = \id_{X_\bu}$, where $p_\bu \colon TX_\bu \to X_\bu$ is the vector bundle projection. It is shown in~\cite[Section~5.4]{hoffman-sjamaar;hamiltonian-stack} that, when $X_\bu$ is a foliation groupoid, the Lie 2-algebra of multiplicative vector fields on $X_\bu$ is equivalent to the Lie algebra $\Vect_0(X_\bu)$ (as a Lie 2-algebra). In particular, we have the following invariance of basic vector fields and forms under Morita equivalence.
\end{remark}

\begin{proposition}\cite[Proposition~5.3.12]{hoffman-sjamaar;hamiltonian-stack} \label{proposition;basicvfandform}
Let $X_\bu$ and $Y_\bu$ be foliation groupoids. A Morita morphism $\phi_\bu\colon X_\bu \to Y_\bu$ induces
\begin{enumerate}
\item an isomorphism $\phi_\bu^* \colon \Vect_0(Y_\bu) \cong \Vect_0(X_\bu)$ of Lie algebras, and
\item an isomorphism $\phi_\bu^* \colon \Omega_0^\bu(Y_\bu) \cong \Omega_0^\bu(X_\bu)$ of complexes.
\end{enumerate}
If $\phi_\bu$ and $\chi_\bu$ are naturally isomorphic Morita morphisms, then $\phi_\bu^*=\chi_\bu^*$.
\end{proposition}

If $v_{X_\bu}\in \Vect_0(X_\bu)$ is a basic vector field on $X_\bu$ and $\omega_{X_\bu}\in \Omega^k_0(X_\bu)$ is a basic $k$-form on $X_\bu$, then the contraction of $\omega_{X_\bu}$ by $v_{X_\bu}$ is defined as
 \[
 \iota_{v_{X_\bu}} \omega_{X_\bu}= (\iota_{v_{X_0}} \omega_{X_0},\iota_{v_{X_1}} \omega_{X_1})\in \Omega_0^{k-1} (X_\bu),
 \]
 where $\iota_{v_{X_i}}\omega_{X_i}:=\iota_{\tilde{v}_{X_i}}\omega_{X_i}$ for any vector field $\tilde{v}_{X_i} \in \Gamma(TX_i)$ lifting $v_{X_i}\in \Gamma (N_i)$. Contraction of basic multivector fields with basic $1$-forms is defined similarly. It is easy to check these definitions do not depend on the choice of lift $\tilde{v}_{X_i}$.

\subsection*{0-Symplectic groupoids}

A 2-form $\omega$ on a smooth manifold $X$ is \emph{presymplectic} if $d\omega=0$ and if the subbundle $\ker\omega\subset TX$ is of constant rank. In this case, $\ker\omega$ can be integrated to a foliation of $X$, called the \emph{null foliation}. We will also denote it by $\ker \omega$.

A 2-form $\omega_{X_\bu} \in \Omega^2_0(X_\bu)$ on a foliation groupoid $X_\bu$ is \emph{0-symplectic} if $\omega_{X_0}$ is presymplectic, and if the leaves of the null foliation $\ker \omega_{X_0}$ are the connected components of the orbits of $X_1$ in $X_0$. In particular, the forms $\omega_0$ and $\omega_1$ are nondegenerate if and only if $X_\bu$ is \'etale. 
Being $0$-symplectic is a Morita invariant condition. As a consequence of the definition, contraction with a 0-symplectic form induces an isomorphism $\Vect_0(X_\bu)\cong \Omega^1_0(X_\bu)$.

\begin{example} \label{example;SymLieGpd}
Let $(X_0,\omega_{X_0})$ be a presymplectic manifold, and let $X_1\rightrightarrows X_0$ be any foliation groupoid with the property that the fibers of the source map $s\colon X_1\to X_0$ are connected, and the property that the $X_1$ orbits in $X_0$ are the leaves of $\ker \omega_0$. Let $\omega_{X_1} = s^*\omega_{X_0}$. Then $(X_\bu,\omega_{X_\bu})$ is $0$-symplectic.

In particular, one may choose $X_\bu$ to be the monodromy groupoid or holonomy groupoid of the foliation $\ker\omega_0$; for a description of these groupoids see~\cite[Chapter~5.2]{moerdijk-mrcun;foliations-groupoids}.
\end{example}

\subsection*{Forms and vector fields on \'etale stacks}

Let $\X$ be an \'etale stack, and let $\B X_\bu \simeq \X$ be a fixed presentation of $\X$ by a foliation groupoid. Then by definition, the complex of differential forms on $\X$ is
\[
\Omega^\bu(\X) =\Omega^\bu_0(X_\bu).
\]
If $\omega_{X_\bu}\in \Omega^\bu_0(X_\bu)$, then we write $\B \omega_{X_\bu}\in \Omega^\bu(\X)$ to denote the corresponding form on $\X$. Similarly, the space of $k$-multivector fields of $\X$ is 
\[
\Vect^k(\X) = \Vect_0^k(X_\bu).
\]
Contraction of vector fields and forms on an \'etale stack $\X$ is defined as contraction of basic vector fields and forms, on the level of presenting groupoids.

Due to Proposition~\ref{proposition;basicvfandform}, these definitions depend on the choice of presentation $\B X_\bu \simeq \X$ only up to isomorphism. 

\begin{remark}
There are equivalent definitions of $\Vect(\X)$ and $\Omega^\bu(\X)$ given in Sections~5.7 and~5.8 of~\cite{hoffman-sjamaar;hamiltonian-stack}, which do not require a choice of fixed presentation $\B X_\bu \simeq \X$. For our purposes it will be enough to work with the description above.
\end{remark}

\subsection*{Symplectic stacks} Let $\X$ be an \'etale stack, and let $\bomega\in \Omega^2(\X)$. Then $(\X,\bomega)$ is \emph{symplectic} if it satisfies the following property: Given a presentation $\B X_\bu\simeq \X$ of $\X$ and a basic form $\omega_{X_\bu} \in \Omega^2_0(X_\bu)$ with $\B \omega_{X_\bu} = \bomega$, then the form $\omega_{X_\bu}$ is 0-symplectic.

\subsection*{Lie 2-groups and their Lie algebras}

A Lie $2$-group $G_\bu$ is a strict group object in the 2-category of Lie groupoids. That is, it is a group object in the 1-category of Lie groupoids and their morphisms. The groupoid structure maps of $G_\bu$ are then Lie group homomorphisms. In this article we consider only Lie $2$-groups with $G_1$ and $G_0$ both abelian.

The Lie algebra of a foliation Lie 2-group $G_\bu$ is defined as the quotient 
\[
\Lie(G_\bu) = \Lie(G_0)/ \Lie(t)( \ker \Lie(s)).
\]
If $\Phi_\bu \colon G_\bu \to H_\bu$ is a morphism of Lie 2-groups which is also a Morita morphism, then it $\Phi_\bu$ is a \emph{Morita morphism of Lie 2-groups}. It induces an isomorphism $\Lie(\Phi_\bu)\colon \Lie(G_\bu) \cong \Lie(H_\bu)$.

A \emph{(strict) action of a Lie 2-group $G_\bu$ on a Lie groupoid $X_\bu$} is a pair of actions $G_0\times X_0 \to X_0$ and $G_1\times X_1\to X_1$ which assemble to a map of Lie groupoids $G_\bu \times X_\bu \to X_\bu$.

\subsection*{Crossed modules}

A \emph{crossed (Lie) module} $\partial\colon H\to G$ is a Lie group homomorphism together with an action of $G$ on $H$, subject to compatibility conditions which are described in Section~6.3 of \cite{hoffman-sjamaar;hamiltonian-stack}. In the context of this article, both $G$ and $H$ will always be abelian, the action of $G$ on $H$ will always be trivial, and these compatibility conditions will be automatically satisfied. A homomorphism of crossed modules $(\partial\colon H\to G) \to (\partial'\colon H'\to G')$ is a pair of Lie group homomorphisms $H\to H'$ and $G\to G'$ which preserve the crossed module structure.

There is an equivalence of 1-categories between the category of Lie 2-groups and the category of crossed modules. In particular, to an abelian crossed Lie module $\partial\colon H\to G$ one associates the Lie $2$-group $H\ltimes G\rightrightarrows G$. The groupoid structure of $H\ltimes G\rightrightarrows G$ is that of an action groupoid, where the action of $H$ on $G$ is given as $h\cdot g = \partial(h)g$, for $h\in H$ and $g\in G$. The group structure on $H\ltimes G$ is the direct product group structure. Conversely, to the abelian Lie $2$-group $G_\bu$ one associates the crossed module $\partial \colon s\n(e) \to G_0$, where $\partial:=t|_{s\n(e)}$ is the target map restricted to arrows whose source is the group unit of $G_0$.

The axioms for an action of a Lie 2-group on a Lie groupoid can be formulated in terms of crossed modules. An action of a crossed module $\partial\colon H\to G$ on a Lie groupoid $X_\bu$ consists of three actions: actions of $G$ on $X_0$ and $X_1$, and an action of $H$ on $X_1$. We will write these actions as
\begin{align}
G\times X_0 & \to X_0 &  (g,x)  \mapsto g\cdot x,  \nonumber \\
G\times X_1 & \to X_1 &  (g,\alpha) \mapsto g\cdot \alpha, \label{xmodact} \\
H\times X_1 & \to X_1 &   (h,\alpha) \mapsto h*\alpha. \nonumber
\end{align}
These actions are subject to compatibility conditions. When $G$ and $H$ are abelian, they are as follows. If $g\in G$, $h\in H$, $x\in X_0$, and $\alpha,\alpha'\in X_1$, then
\begin{equation} \label{equation;compatcross}
\begin{split}
&g\cdot u(x) = u(g\cdot x);\quad g\cdot s(\alpha) = s(g\cdot \alpha);\quad g\cdot t(\alpha) = t(g\cdot \alpha);\\
&g\cdot(\alpha\circ \alpha') = (g\cdot \alpha)\circ (g\cdot \alpha');\\
&s(h*\alpha) =s(\alpha);\quad t(h*\alpha) = \partial(h)\cdot t(\alpha);\\
&h*\alpha = (h* u(t(\alpha)))\circ \alpha = (\partial(h)\cdot \alpha) \circ (h *u(s(\alpha))).
\end{split}
\end{equation}
We will freely pass between the languages of Lie 2-groups and crossed modules in what follows.

\subsection*{Lie group stacks and their Lie algebras}

In the context of this article, a \emph{Lie group stack} $\G$ will denote a (strict) group object internal to the category of differentiable stacks. Since any $\G$ we consider will be \'etale and have a connected atlas, by the main theorem of \cite{trentinaglia-zhu;strictification} there is no loss in generality in assuming that $\G$ is strict. The 2-functor $\B$ then takes Lie 2-groups to Lie group stacks. 

Equivalence of Lie group stacks is defined in Section~6.5 of \cite{hoffman-sjamaar;hamiltonian-stack}. If $G_\bu$ and $G_\bu'$ are Lie 2-groups and $\B G_\bu$ is equivalent to $\B G_\bu'$ as a Lie group stack, then $G_\bu$ is Morita equivalent to $G_\bu'$ as a Lie 2-group~\cite[Proposition~6.6.2]{hoffman-sjamaar;hamiltonian-stack}. A \emph{presentation} of a Lie group stack $\G$ by a Lie 2-group $G_\bu$ is an equivalence $\B G_\bu \simeq \G$ of Lie group stacks. If $\G$ is presented by the Lie 2-group associated to the crossed module $\partial \colon H \to G$, then we also say $\G$ is presented by this crossed module.

Given an \'etale Lie group stack $\G$ with a fixed presentation by a Lie 2-group $G_\bu$, define $\Lie(\G)= \Lie(G_\bu)$. Since Morita equivalences of Lie 2-groups induce isomorphisms of the associated Lie algebras, the Lie algebra $\Lie(\G)$ depends on the choice of the presentation $G_\bu$ only up to isomorphism. We will typically assume a presentation of $\G$ has been fixed. An equivalence $\stackmorphism{\Phi}\colon \G \to \G'$ of stacky Lie groups induces an isomorphism of their Lie algebras $\Lie(\stackmorphism{\Phi})\colon \Lie(\G)\cong \Lie(\G')$.

Axioms for an action of a Lie group stack on a differentiable stack are as in Section 3.2 of \cite{bursztyn-noseda-zhu;principal-stacky-groupoids}. However, by Theorem \ref{appAthm} below, in the context of this article we can always assume an action of a Lie group stack on an \'etale stack comes from a strict action on the level of their presenting groupoids.

\subsection*{Hamiltonian actions on 0-symplectic groupoids}

A \emph{Hamiltonian Lie groupoid} is a tuple $(X_\bu,\omega_{X_\bu},G_\bu,\mu_{X_\bu})$, where $(X_\bu,\omega_{X_\bu})$ is a 0-symplectic groupoid, $G_\bu$ is a foliation Lie 2-group with $G_0$ connected, and $\mu_{X_\bu}\colon X_\bu \to \Lie(G_\bu)^*$ is a moment map, satisfying the following conditions:
\begin{enumerate}
\item $G_\bu$ acts on $X_\bu$
\item If $v \in \g_0=\Lie(G_0)$, then contracting the presymplectic form $\omega_{X_0}$ with the fundamental vector field $v_{X_0}$ gives the $v$-component of $d\mu_{X_0}$:
\[
\langle d\mu_{X_0}, v \rangle = \iota_{v_{X_0}} \omega_{X_0}.
\]
\item The moment map $\mu_{X_0}$ intertwines the $G_0$ action on $X_0$ and the coadjoint action of $G_0$ on $\Lie(G_\bu)^*\subseteq \Lie(G_0)^*$.
\end{enumerate}
A pair $(\Phi_{X_\bu} \colon X_\bu \to X_\bu', \Phi_{\G_\bu} \colon G_\bu \to G_\bu')$ is a Morita morphism of the Hamiltonian groupoids $(X_\bu,\omega_{X_\bu},G_\bu,\mu_{X_\bu})$ and $(X_\bu',\omega_{X'_\bu},G'_\bu,\mu_{X'_\bu})$ if  $\Phi_{G_\bu}$ is a Morita morphism of Lie 2-groups, and $\Phi_{X_\bu}$ is a Morita morphism of Lie groupoids which preserves the symplectic form and moment maps.

\begin{example} \label{example;HamLieGpd}
Let $(X_0,\omega_0)$ be a presymplectic manifold. Assume $G_0$ is a compact torus, and let 
\[
\mathfrak{n} = \{ v \in \Lie(G_0) \mid \iota_{v_{X_0}} \omega_0 =0 \}.
\]
Given a map $\mu_{X_0}\colon X_0 \to (\Lie(G_0) /\mathfrak{n})^*$, assume that $(X_0,\omega_{X_0},G_0,\mu_{X_0})$ satisfies the following:
\begin{enumerate}
\item $G_0$ acts on $X_0$;
\item $\langle d \mu_{X_0} v \rangle = \iota_{v_{X_0}} \omega_{X_0}$ for all $v\in \Lie(G_0)$; 
\item the map $\mu_{X_0}$ is invariant with respect to the $G_0$ action on $X_0$.
\end{enumerate}
The tuple $(X_0,\omega_{X_0},G_0,\mu_{X_0})$ is known as a \emph{presymplectic Hamiltonian $G_0$-manifold}.
In~\cite[Theorem~7.2.1]{hoffman-sjamaar;hamiltonian-stack}, it was shown how to build a Hamiltonian Lie groupoid from a presymplectic Hamiltonian manifold. Here we describe that construction, in a simple case.

Let $(X_\bu,\omega_{X_\bu})$ be a 0-symplectic groupoid with connected $s$-fibers integrating $(X_0,\omega_{X_0})$, as in Example~\ref{example;SymLieGpd}. Let $N$ be the simply connected Lie group with Lie algebra $\mathfrak{n}$. The inclusion $\mathfrak{n}\hookrightarrow \Lie(G_0)$ integrates to a map of Lie groups $\partial\colon N \to G_0$, which is a crossed module. Let $G_\bu$ be the associated Lie 2-group; then $\Lie(G_\bu) = \Lie(G_0)/\mathfrak{n}$. Then, there is a naturally defined action on $G_\bu$ on $X_\bu$. Putting $\mu_{X_1}=s^*\mu_{X_0}$, the tuple
\[
(X_\bu,\omega_{X_\bu}, G_\bu,\mu_{X_\bu})
\]
is a Hamiltonian Lie groupoid.
\end{example}

\subsection*{Hamiltonian actions on stacks} We now recall the intrinsic definition of a Hamiltonian stack. Since we will typically employ Theorem \ref{appAthm} below, we omit certain details.

A \emph{(connected) Hamiltonian stack} is a tuple $(\X,\bomega,\G,\bmu)$, where $(\X,\bomega)$ is a symplectic stack, $\G$ is an \'etale Lie group stack with $\G^{coarse}$ connected, and $\bmu\colon \X\to \Lie(\G)^*$ is a moment map, satisfying the following conditions:
\begin{enumerate}
\item $\G$ acts on $\X$
\item If $v \in \Lie(\G)$, then contracting the symplectic form $\bomega$ with the fundamental vector field $v_{\X}$ gives the $v$-component of $d\bmu$: 
\[
\langle d\bmu, v \rangle = \iota_{v_{\X}} \bomega.
\]
\item The moment map $\bmu$ intertwines the $\G$ action on $\X$ and the coadjoint action of $\G$ on $\Lie(\G)^*$.
\end{enumerate}
The fundamental vector field $v_\X\in \Vect(\X)$ and the coadjoint action of $\G$ on $\Lie(\G)^*$ are defined in general in Sections~6.9 and~6.10 of \cite{hoffman-sjamaar;hamiltonian-stack}. In the context of compact toric symplectic stacks the coadjoint action is always trivial. 

A pair $(\stackmorphism{\Phi}_\X\colon \X \to \X', \stackmorphism{\Phi}_\G \colon \G \to \G')$ is an equivalence of the Hamiltonian stacks $(\X,\bomega,\G,\bmu)$ and $(\X',\bomega',\G',\bmu')$ if $\stackmorphism{\Phi}_\G$ is an equivalence of Lie group stacks and $\stackmorphism{\Phi}_\X$ is an equivalence of stacks which preserves the symplectic form and moment map.

If $(X_\bu,\omega_{X_\bu},G_\bu,\mu_{X_\bu})$ is a Hamiltonian groupoid, then $(\B X_\bu, \B \omega_{X_\bu}, \B G_\bu, \B \mu_{X_\bu})$ is a Hamiltonian stack. The following gives a partial converse. It summarizes Theorem~6.8.1 and Theorem~7.3.2 of \cite{hoffman-sjamaar;hamiltonian-stack}.

\begin{theorem} \label{appAthm}
Let $\G$ be an \'etale Lie group stack which acts on an \'etale stack $\X$. Let $G_\bu$ be a Lie 2-group presentation of $\G$, and assume $G_0$ is connected. Then
\begin{enumerate}
\item \label{firstpresentactionitem} There exists a presentation $X_\bu$ of $\X$ and a free action of $G_0$ on $X_0$, so that the diagram 2-commutes:
\begin{equation}
\label{presactiondiag}
\begin{tikzcd}
G_0\times X_0 \arrow[r] \arrow[d] & X_0 \arrow[d] \\
\G \times \X \arrow[r] & \X.
\end{tikzcd}
\end{equation}
Here the horizontal arrows are the action maps and the vertical arrows are the atlases.
\item \label{presentactionitem} Assume $X_\bu$ is a presentation of $\X$ and there is a (not necessarily free) action of $G_0$ on $X_0$ so that the diagram \eqref{presactiondiag} 2-commutes. Then there is a (strict) action of $G_\bu$ on $X_\bu$ which, on the manifolds of objects, is the original action of $G_0$ on $X_0$. 
\item Assume $(\X,\bomega,\G,\bmu)$ is a Hamiltonian stack with $\B X_\bu\simeq \X$ and $ \B G_\bu\simeq \G$, and assume the action of $\G$ on $\X$ lifts to a strict action of $G_\bu$ on $X_\bu$, as in part \eqref{presentactionitem}. Then there is a unique 0-symplectic form $\omega_{X_\bu}\in \Omega^2_0(X_\bu)$ and moment map $\mu_{X_\bu}\colon X_\bu \to \Lie(G_\bu)^*$ so that
\[
(\X,\bomega,\G,\bmu)\simeq (\B X_\bu, \B \omega_{X_\bu}, \B G_\bu, \B \mu_{X_\bu})
\]
as Hamiltonian stacks.
\end{enumerate}
\end{theorem}
In the situation that $\B X_\bu\simeq \X$ and $\B G_\bu \simeq \G$ satisfy the conclusion of Theorem~\ref{appAthm}~\eqref{presentactionitem}, then we will say that the action of $G_\bu$ on $X_\bu$ \emph{presents} the action of $\G$ on $\X$.

\begin{remark}
It is natural to ask if Theorem~\ref{appAthm} can be strengthened the following form: Given an action of an \'etale Lie group stack $\G$ on an \'etale stack $\X$, there exist \'etale groupoids $G_\bu$ and $X_\bu$, presentations $\B G_\bu \simeq \G$ and $\B X_\bu \simeq \X$, and an action of $G_\bu$ on $X_\bu$ which presents the action of $\G$ on $\X$. The proof of~\cite[Theorem~6.8.1]{hoffman-sjamaar;hamiltonian-stack} does not guarantee this is possible, but no counterexamples to this stronger claim are known.
\end{remark}
%
%

\section{Toric Symplectic Stacks}

\label{section3}

In this section, we define stacky tori and toric symplectic stacks.

\begin{definition}
A \emph{$2$-torus} is a Lie $2$-group which is Morita
equivalent to a foliation $2$-group $G_\bu$, with the property that
$G_0$ is a torus and $\pi_0(G_1)$ is finitely generated. A \emph{stacky torus} is an \'etale Lie group stack equivalent to $\B G_\bu$, where $G_\bu$ is a $2$-torus. A \emph{quasilattice} $\partial\colon Q\to V$ is a crossed Lie module, where $V\cong \R^n$, $Q$ is a finitely generated abelian group, and $\partial(Q)$ spans $V$ as a vector space over $\R$.
\end{definition}

A quasilattice $\partial\colon Q\to V$ is a crossed module, and so by definition it comes with an action of $V$ on $Q$. However, because $V$ is connected and abelian, this action is automatically trivial~\cite[Lemma~6.11.1]{hoffman-sjamaar;hamiltonian-stack}. Because $Q$ is assumed to be finitely generated, we have $\operatorname{Ext}_\Z^1(\partial(Q),\ker\partial)=0$. Let us fix a splitting $Q  = \ker \partial\times \partial(Q)$, once and for all.
 
 If $\partial \colon Q \to V$ is a quasilattice, then the associated Lie $2$-group $Q\ltimes V\rightrightarrows V$ is a $2$-torus. Indeed, if $\Lambda\subset \partial(Q)$ is a lattice in $V$, then the quotient map 
 \[
 (\ker\partial \times \partial(Q)) \ltimes V \to (\ker \partial \times \partial(Q)/\Lambda)\ltimes V/\Lambda
 \]
  is a Morita equivalence. Conversely, if $\G$ is a stacky torus, recall from~\cite[Definition~5.1]{trentinaglia-zhu;strictification} that $\pi_1(\G)$ denotes the group of based loops $S^1\to \G$, modulo homotopy. Due to \cite[Proposition~6.11.3]{hoffman-sjamaar;hamiltonian-stack}, the stacky torus $\G$ is presented by the quasilattice $\pi_1(\G) \to \Lie(\G)$.

\begin{definition}
\label{toricdef}
Let $\G$ be a stacky torus, and let $(\X,\bomega,\G,\bmu)$ be a Hamiltonian $\G$-stack. It is a \emph{compact toric symplectic stack} if the following hold:
\begin{enumerate}
\item \label{cond1} $2\dim\G=\dim \X$.
\item \label{cond4} $\X^{coarse}$ is compact and connected.
\item \label{cond5} The coarse reduced space $(\bmu^{coarse})\n(p)/\G^{coarse}$ is homeomorphic to a point, for all $p\in \bmu(\X)$.
\item \label{cond6} If $p\in \bmu(\X)\subset \Lie(\G)^*$ is a regular value of $\bmu$, and 
\[
x\colon \star \to \bmu\n(p) := \X\times_{\Lie(\G)^*} \{p\}
\]
 is a categorical point of $\X$, then the action map $\G\times \star \hookrightarrow \G \times \bmu\n(p) \to \bmu\n(p)$ is an equivalence of stacks.
\item \label{cond2} The set of regular points of $\bmu^{coarse}$ is open and dense in $\X^{coarse}$.
\end{enumerate}
\end{definition}

\begin{remark}
Conditions \eqref{cond1}, \eqref{cond4}, and \eqref{cond2} are generalizations of the requirements placed on a compact toric symplectic manifold. If $(M,\omega,G,\mu)$ is a toric symplectic manifold and $p\in \mu(M)$, then $G$ acts transitively on $\mu\n(p)$. If $p$ is a regular value of $\mu$, this action is also free. Conditions \eqref{cond5} and \eqref{cond6} are stacky reformulations of these facts. In view of~\cite[Section~9]{hoffman-sjamaar;hamiltonian-stack}, Condition~\eqref{cond6} is equivalent to the requirement that, for each regular value $p$ of $\bmu$, the stacky reduced space $\bmu\n(p)/\G$ is equivalent (as a stack) to a point.

If \eqref{cond5} and \eqref{cond6} were not included as part of the definition of a compact toric symplectic stack, then a classification along the lines of Theorem \ref{maintheorem} would be difficult: since non-Hausdorff manifolds can be thought of as stacks, one could form a stack satisfying \eqref{cond1}, \eqref{cond4}, and \eqref{cond2} by taking a compact toric symplectic manifold $(M,\omega, G,\mu)$ and making it into a non-Hausdorff manifold by making a fixed point of the $G$ action into a double point. Or, more generally, by doubling any $G$-orbit.
\end{remark}

\begin{lemma} \label{factlemma}
Let $(\X,\bomega,\G,\bmu)$ be a compact toric symplectic stack. Fix a presentation $X_\bu$ of $\X$. Let $p\in \bmu(\X)$ and $v\in \Lie(\G)$. If $x,y\in \mu_{X_0}\n(p)$, then 
\[
\langle d\mu_{X_0}|_x, v \rangle=0 \text{ if and only if } \langle d\mu_{X_0}|_y, v \rangle =0.
\]
\end{lemma}

\begin{proof}
%
Let $x'\in \X^\coarse$ be the image of $x$ under $X_0\to X_0/X_1 \cong \X^{coarse}$; then $\langle d\mu_{X_0}|_x, v \rangle$ vanishes if and only if $x'$ is a regular point of $\langle \bmu^{coarse}, v\rangle$.
 It therefore sufficient to prove the claim for a single choice of groupoid presentation $X_\bu$.
Fix a presentation $G_\bu$ of $\G$, with $G_0$ connected. Following parts  \eqref{firstpresentactionitem} and \eqref{presentactionitem} of Theorem \ref{appAthm}, let us choose $X_\bu$ to be a presentation of $\X$ so that the action of $\G$ on $\X$ is presented by an action of $G_\bu$ on $X_\bu$. Let $x,y\in \mu_{X_0}\n(p)$. By Definition \ref{toricdef} \eqref{cond5}, there is some $g\in G_0$ so that $g\cdot x$ is in the $X_1$-orbit of $y$. Without loss of generality we can assume $g\cdot x=y$. Then, since $G_0$ acts parallel to fibers of $\mu_{X_0}$, we have $\langle d\mu_{X_0} |_x, v \rangle=0$ if and only if $\langle d\mu_{X_0}|_y, v \rangle =0$
\end{proof}

\section{Local normal form}

\label{section4}

In this section we show that compact toric symplectic stacks are locally presented by an action groupoid, for an action of a discrete abelian group on a piece of a toric symplectic manifold. This local structure informs many of the proofs which follow in Sections \ref{section5}-\ref{section9}. Most of the section is devoted to proving Proposition \ref{initiallocal}; this result is combined with the local normal form for toric symplectic manifolds to give the desired Theorem \ref{localnormalform}.

\begin{proposition}
\label{initiallocal}
Let $(\X,\bomega,\G,\bmu)$ be a compact toric symplectic stack, and fix a quasilattice $\partial\colon Q\to \R^n$ presenting $\G$. For $p\in \bmu(\X)$, there is an open neighborhood $U\subset \bmu(\X)$ of $p$, and a lattice $\Lambda\subset \partial(Q)$ in $\R^n$ so that
 so that
\[
(\bmu\n(U),\bomega,\G,\bmu) \simeq ( \B Z_\bu , \B\omega_{Z_\bu} , \B G_\bu, \B \mu_{Z_\bu})
\]
as Hamiltonian $\G$-stacks. Here, $G_0 = \R^n/\Lambda \cong \T^n$ and $G_1= (\ker\partial\times \partial(Q)/\Lambda)\ltimes G_0$ is an \'etale presentation of $\G$. The groupoid $Z_\bu$ is \'etale and $Z_0$ is connected. The tuple
$(Z_\bu,\omega_{Z_\bu},G_\bu,\mu_{Z_\bu})$ is a Hamiltonian $G_\bu$ groupoid, and  $G_0$ acts effectively on $Z_0$. 
\end{proposition}

\begin{proof}

The structure of the proof is as follows. After some initial setup, we prove four lemmas. In Lemma \ref{ESlemma} we build an \'etale presentation of the preimage under $\bmu$ of some small open set $U\subset \bmu(\X)$. In Lemma \ref{vfcompletelemma} we show that the Hamiltonian vector fields of $\bmu$ can be integrated on this presentation (possibly after shrinking $U$). The action of $\R^n$ coming from these vector fields descends to an effective action of a torus according to Lemma \ref{descendlemma}; for this we use Lemma \ref{connected}, which says that the regular values of $\bmu(\X)$ are connected in $\bmu(\X)$. The proposition then quickly follows.

Fix an \'etale presentation $\tilde{G}_\bu$ of $\G$, with $\tilde{G}_0=\R^n/\Z^n =\T^n$. Let $X_\bu$ be a presentation of $\X$ so that the action of $\G$ on $\X$ comes from an action of $\tilde{G}_\bu$ on $X_\bu$, with the action of $\tilde{G}_0$ on $X_0$ free, as in Theorem \ref{appAthm}. Let $x\in \mu_{X_0}\n (p)$. Choose a basis $(v^i)_{i\in[n]}$ of $\g=\Lie(\G)\cong \R^n$ so that $\langle d\mu_{X_0}|_x, v^i \rangle=0$ for $i=1,\dots, k$, and so that the covectors $\langle d\mu_{X_0}|_x, v^i \rangle$, $i=k+1,\dots, n$ are linearly independent. We can assume that the vectors $v^i$ are contained in the cocharacter lattice $\Hom(S^1,\tilde{G}_0)\cong \Z^n\subset \g$, for $i=k+1,\dots, n$.
Decompose $\g$ as
\[
\g = \mathfrak{s} + \mathfrak{k}
\]
where $\mathfrak{s}=\operatorname{span}\{v^i\mid i=1,\dots, k \}$ and $\mathfrak{k}=\operatorname{span}\{ v^i\mid i=k+1,\dots, n\}$. Let $S \cong \R^{k}$ be the simply connected group with Lie algebra $\mathfrak{s}$, and let $K\hookrightarrow G$ be the connected subgroup of $G$ with Lie algebra $\mathfrak{k}$. Then $K$ is a compact torus of dimension $n-k$, and we have an immersion $S\to \tilde{G}_0$ induced by the inclusion $\mathfrak{s}\hookrightarrow \mathfrak{g}$. Note that $x$ is a fixed point of the Hamiltonian action of $S$ on $X_0$.

Let $W$ be a small open neighborhood of $x$ on which $\mathfrak{k}$ acts nontangently to the null foliation $\ker \omega_{X_0}$. By replacing $W$ with $S\cdot W$, we can assume that $W$ is saturated by $S$ orbits. In doing so, we still have that $\mathfrak{k}$ acts nontangently to the null foliation on all of $W$. 

Let $D$ be the subbundle of $TW$ spanned by $\ker \omega_{X_0}$ and the fundamental vector fields of $\mathfrak{k}$. Let $\Sigma$ be a transversal to $D$ passing through $x$. By shrinking $W$ we can assume $\Sigma$ is a connected immersed submanifold of $X_0$.

\begin{lemma} \label{ESlemma} The action map
\[
a: K\times \Sigma \to \mu_{X_\bu}\n(\mu_{X_0}(W))
\]
is essentially surjective. So, the pullback groupoid $Y_\bu$ over $Y_0:= K\times \Sigma$ is an \'etale presentation of $\bmu\n(\mu_{X_0}(W))=\bmu\n(\mu_{Y_0}(Y_0))$.
\end{lemma}

\begin{proof}
The map $a$ factors as
\[
\begin{tikzcd}
K\times \Sigma \arrow[r]&  K\cdot W \arrow[r, hookrightarrow, "\iota"]& \mu_{X_\bu}\n(\mu_{X_0}(W))
\end{tikzcd}
\]
We first show the inclusion $\iota$ is essentially surjective; for this, we need to check that
\[
t\circ \pr_2 \colon (K\cdot W)\times_{a,X_0,s} X_1 \to \mu_{X_0}\n(\mu_{X_0}(W))
\]
is a surjective submersion. It is a submersion onto its image because $K\cdot W$ is open in $X_0$. To check it is surjective, let $y\in \mu_{X_0}\n(\mu_{X_0}(W))$.  By Definition \ref{toricdef} \eqref{cond5}, there is some $g\in \tilde{G}_0$ and $x'\in W$ so that $g\cdot x'$ is in the $X_1$-orbit of $y$. Because $W$ is saturated by $S$ orbits, we can assume $g\in K$. So there is $g\cdot x'\in K\cdot W$ which is in the $X_1$-orbit of $y$, as desired.

Now, consider the pullback groupoid $\mu_{X_1} \n(\mu_{X_0}(W)|_{K\cdot W}\rightrightarrows K\cdot W$. The map $K\times \Sigma\to K\cdot W$ is a complete transversal to the null foliation $\ker \omega_{X_0}$ of $K\cdot W$, and so this map is essentially surjective onto $\mu_{X_\bu}\n(\mu_{X_0}(W)|_{K\cdot W}$. Therefore, the composition $K\times \Sigma\to \mu_{X_\bu}\n(\mu_{X_0}(W))$ is essentially surjective.
\end{proof}

We will now consider the $0$-symplectic groupoid $(Y_\bu,\omega_{Y_\bu})$ constructed in Lemma \ref{ESlemma}, where $\omega_{Y_\bu}$ is the pullback of $\omega_{X_\bu}$ under the natural map $\iota_\bu \colon Y_\bu \to X_\bu$. Note that $(Y_0,\omega_{Y_0})$ is symplectic because $Y_\bu$ is \'etale.  Consider the moment map $\mu_{Y_\bu}= \mu_{X_\bu}\circ \iota_\bu$. We will abbreviate $x = (e,x)\in Y_0$. For $v\in \mathfrak{s}$, consider the Hamiltonian vector field $v_{Y_0}$ on $Y_0$, defined by $\iota_{v_{Y_0}} \omega_{Y_0} = \langle d\mu_{Y_0}, v\rangle$. 

\begin{lemma}
\label{vfcompletelemma}
There is an open neighborhood $Z_0$ of $K\cdot x$ in $Y_0$, on which the vector fields $v_{Y_0}$ are complete, for all $v\in \mathfrak{s}$.
\end{lemma}
 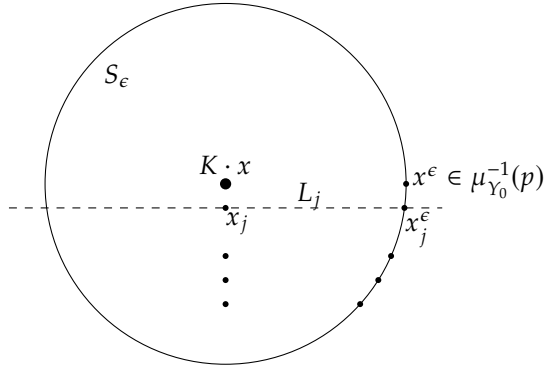
\begin{figure}[htp]
\centering
\begin{tikzpicture}[scale=1.6]	
		\draw (0,0) circle (1.5);
		\node at (0,0)[circle,fill,inner sep=1.5pt]{};
		\draw (0,0) node[align=left, above] {$K\cdot x$};
		\node at (0,-1)[circle,fill, inner sep=.8pt]{};
		\node at (0,-.2)[circle,fill, inner sep=.8pt]{};
		\draw (.1,-.5) node[align=left, above] {$x_j$};
		\node at (0,-.6)[circle,fill, inner sep=.8pt]{};
		\node at (0,-.8)[circle,fill, inner sep=.8pt]{};
		\draw[dashed] (-1.8,-.2)--(1.8,-.2);
		\draw (.7,-.3) node[align=left, above] {$L_j$};
		\draw (-.9,.9)  node[align=left] {$S_\epsilon$};
		\node at (1.486,-.2)[circle,fill, inner sep=.8pt]{};
		\draw (1.6,-.6)node[align=left, above] {$x_j^\epsilon$};
		\node at (1.5,0)[circle,fill, inner sep=.8pt]{};
		\draw (2.1,-.2)node[align=left, above] {$x^\epsilon\in \mu_{Y_0}\n(p)$};
		\node at (1.375,-.6)[circle,fill, inner sep=.8pt]{};
		\node at (1.269,-.8)[circle,fill, inner sep=.8pt]{};
		\node at (1.118,-1)[circle,fill, inner sep=.8pt]{};
	\end{tikzpicture}	
	\caption{Proof of Lemma \ref{vfcompletelemma}. As $j\to \infty$, we have $x_j^\epsilon\to x^\epsilon$.}
\label{fig 0}
\end{figure}
\begin{proof} 
Put a $K$-invariant metric on $Y_0$. For small $\epsilon>0$, let $S_\epsilon$ be the set of points of distance $\epsilon$ from $K\cdot x\cong \T^{n-k}$, and let $B_\epsilon=\cup_{0\le \epsilon'< \epsilon} S_{\epsilon'}$. Fix a small $\delta>0$ so that $B_\delta$ is a tubular neighborhood of $K\cdot x$. Let $\g\cdot x$ denote the orbit of $x$ under the $\g$ action on $Y_0$.

Assume that in any open neighborhood $B$ of $K\cdot x$ in $Y_0$, there is a $v\in \g$ so that the trajectory of the Hamiltonian vector field $v_{Y_0}$ through a point $y\in B$ is not complete. Then, for any sufficiently large $j\in \Z_{>0}$, there is some $x_j\in B_{1/j}$ so that the submanifold $L_j=\g \cdot x_j$ escapes away from $K\cdot x$ and crosses $S_\delta$. For any given $0<\epsilon< \delta$, for large enough $j$, the intersection $L_j\cap S_\epsilon$ is nonempty. Let $x_j^\epsilon\in L_j\cap S_\epsilon$. For a fixed $\epsilon$, since $S_\epsilon$ is compact we can assume without loss of generality that there is some $x^\epsilon\in S_\epsilon$ so that $x_j^\epsilon\to x^\epsilon$ as $j\to \infty$. But,
\[
\mu_{Y_0}(x_j^\epsilon)=\mu_{Y_0}(L_j)=\mu_{Y_0}(x_j)\to \mu_{Y_0}(x)=p \text{ as }j\to \infty
\]
On the other hand,
\[
\mu_{Y_0}(x_j^\epsilon) \to \mu_{Y_0}(x^\epsilon) \text{ as }j \to \infty.
\]
So we then have $\mu_{Y_0}(x^\epsilon)=p$. 
 Because the metric on $Y_0$ is $K$-invariant, we have $K\cdot x^\epsilon\subset S_\epsilon$. From Lemma \ref{factlemma} and the fact that $ \mathfrak{s}\cdot x=x$ we know $\mathfrak{s}\cdot x^{\epsilon} = x^\epsilon$. So, $ \mathfrak{g}\cdot x^\epsilon = \mathfrak{k}\cdot x^\epsilon = K\cdot x^\epsilon \subset S_\epsilon$. 

By Definition \ref{toricdef}  \eqref{cond5}, for each $\epsilon, \epsilon'< \delta $, there is an arrow $\alpha\in Y_1$ with $s(\alpha)=x^\epsilon$, and $t(\alpha)\in K\cdot x^{\epsilon'}$. Since $Y_\bu$ is \'etale, the $s$-fiber $s\n(x^\epsilon)$ is a discrete set, it therefore has uncountably many components, one for each $\epsilon'$. This contradicts the assumption that $Y_\bu$ is second countable. Therefore, for some large $j$, the trajectories through points in $B_{1/j}$ are all contained in the tubular neighborhood $B_\delta$. Let $Z_0\subset Y_0$ be the neighborhood of $K\cdot x$ which is saturated by the $\mathfrak{g}$-trajectories through the points in $B_{1/j}$.
\end{proof}

By Lemma \ref{vfcompletelemma}, there is a Hamiltonian action of $S\times K$ on $(Z_0,\omega_{Z_0}=\omega_{Y_0}|_{Z_0})$, with moment map $\mu_{Z_0} = \mu_{Y_0}|_{Z_0}$. Notice $Z_0$ is of the form $Z_0=K\times \Sigma'$, where $\Sigma'$ is open in $\Sigma$. Then $Z_\bu$ is an \'etale presentation for  $\bmu\n(\mu_{Z_0}(Z_0))$, similarly to in Lemma \ref{ESlemma}. 

If $p\in \bmu(\X)$, then $\mu_{Z_0}\n(p)$ is a closed subset of the compact set $\overline{B_{\delta}}$, which is the closure of the tubular neighborhood considered in the proof of Lemma \ref{vfcompletelemma}. So $\mu_{Z_0}\n(p)$ is compact. In particular, if $p$ is a regular value of $\bmu$, then a component of $\mu_{Z_0}\n(p)$ is a compact submanifold of $Z_0$. By Definition \ref{toricdef} \eqref{cond5}, the action of $\R^n$ is transitive on each component of $\mu_{Z_0}\n(p)$. Therefore, components of regular fibers of $\mu_{Z_0}$ are homeomorphic to $\T^n$.

We then have an action of $\R^n$ on $Z_0$, so that the following diagram 2-commutes:
\[
\begin{tikzcd}
\R^n \times Z_0 \arrow[r] \arrow[d] & Z_0 \arrow[d] \\
\G \times \X \arrow[r] & \X.
\end{tikzcd}
\]
Consider the quasilattice $\partial \colon Q\to \R^n$ presenting $\G$. Let $\omega_{Z_1} = s^* \omega_{Z_0}$, and let $\mu_{Z_1} = s^* \mu_{Z_0}$. By Theorem \ref{appAthm}, there is a strict Hamiltonian action of $\partial\colon Q\to \R^n$ on $(Z_\bu, \omega_{Z_\bu})$ which presents the action of $\G$ on $\X$. Its moment map is $\mu_{Z_\bu}$.

\begin{lemma}
\label{connected}
The set of regular values of $\bmu$ is connected in $\Delta=\bmu(\X)$.
\end{lemma}

\begin{proof}
For $p\in \Delta$ and $x\in \mu_{X_0}\n(p)$, define
\[
\nu(p) = \dim \ker(d\mu_{X_0}|_x).
\]
By Lemma \ref{factlemma} this quantity does not depend on the choice of $x$ or the choice of presentation $X_\bu$ of $\X$. Let $\Delta_k = \{p\in \Delta \mid \nu(p) = k\}$. For each $p\in \Delta$, we will show that there is a connected open neighborhood $V_p$ of $p$ in $\Delta$ with the following properties: 
\begin{enumerate}
\item\label{inductivecond1} If $q\in V_p$, then $\nu(q)\le \nu(p)$.
\item\label{inductivecond2} If $k\ge \nu(p)>0$, then the set $V_p \backslash \Delta_{k}$ is connected. 
\end{enumerate}

For $p\in \Delta$, construct $V_p$ as follows: Let $x\in \mu_{X_0}\n(p)$, and form a Hamiltonian manifold $(Z_0,\omega_{Z_0},S\times K,\mu_{Z_0})$ centered at $x$, as in the discussion following Lemma \ref{vfcompletelemma}. By shrinking $Z_0$, we can assume that, for any $y\in Z_0$, we have $\nu(\mu_{Z_0}(y))\le \nu(p)$, this being an open condition. If $p$ is a regular value of $\mu$, then let $V_p= \mu_{Z_0}(Z_0)$. Otherwise, consider the action of $S\times K$ on $Z_0$. The action of $K\cong \T^{n-\nu(p)}$ is free and the action of $S\cong \R^{\nu(p)}$ fixes $x$. Since $p$ is not a regular value, we have $\dim K<\frac{1}{2} \dim \X$. Write $\mu_{Z_0} = \mu_{Z_0,S} \oplus \mu_{Z_0,K} \colon Z_0\to \mathfrak{s}^*\oplus \mathfrak{k}^*$. 

Let $M$ be the component of $x$ in $\mu_{Z_0,K}\n(\mu_{Z_0,K}(x))$. By the remarks preceding this lemma $M$ is compact. We will show $M\backslash (\mu_{Z_0}\n (\Delta_{\nu(p)}))$ is connected. The reduced space $M/K$ is a (nonempty) symplectic manifold of dimension $\dim\X-2\dim K\ge 2$. The $S$ action on $M$ descends to a Hamiltonian action on $M/K$, with moment map $\mu_{M/K,S}\colon M/K\to \mathfrak{s}^*$. The subspace $\mu_{Z_0}\n (\Delta_{\nu(p)})/K$ of $M/K$ is precisely the locus $(M/K)^S$ which is stabilized by the action of $S$. The space $(M/K)^S$ is totally disconnected, as follows. If $N\subset (M/K)^S$ is connected, then since $d\mu_{M/K,S}|_{(M/K)^S} = 0$, the map $\mu_{M/K,S}$ is constant on $N$. So $N$ is contained in a single fiber of $\mu_{M/K,S}$. By Definition \ref{toricdef} \eqref{cond5} and the fact that $Z_\bu$ is \'etale, $N$ must be a single point. Therefore, $(M/K)^S$ is compact and totally disconnected, and $\dim M/K\ge 2$. The complement of a compact, totally disconnected set in a manifold of dimension $\ge 2$ is connected \cite{thurston;alexander}. Therefore, the complement $(M\backslash \mu_{Z_0}\n (\Delta_{\nu(p)})) / K $ of $(M/K)^S$ in $M/K$ is connected. Because $K$ is connected, we have then found that in $M\backslash (\mu_{Z_0}\n (\Delta_{\nu(p)}))$ is connected.

Now, consider a connected open neighborhood $\tilde{V}$ of $M$ in $Z_0$. Then $\tilde{V}\backslash \mu_{Z_0}\n(\Delta_{\nu(p)})$ is an open neighborhood of $M\backslash (\mu_{Z_0}\n (\Delta_{\nu(p)}))$. Since $M\backslash (\mu_{Z_0}\n (\Delta_{\nu(p)}))$ is connected, by possibly shrinking $\tilde{V}$, we can assume $\tilde{V}\backslash \mu_{Z_0}\n(\Delta_{\nu(p)})$ is connected. Let $V_p = \mu_{Z_0}(\tilde{V})$. The set $V_p$ evidently has the desired properties \eqref{inductivecond1} and \eqref{inductivecond2}.

To show that $\Delta_0=\Delta\backslash (\Delta_1\cup\cdots \cup \Delta_n)$ is connected, we proceed by induction. First, $\Delta$ is connected by Definition \ref{toricdef} \eqref{cond4}. Now, let $k\ge 1$. Assume that $\Delta_{\le k}:=\Delta \backslash \left(\bigcup_{j>k} \Delta_j\right)$ is connected. The set $\Delta_{<k}:= \Delta_{\le k}\backslash \Delta_k$ has an open cover by $\{V_p\backslash \Delta_k\}_{p\in \Delta_{\le k}}$. By construction, each of these sets is connected. And, by Definition \ref{toricdef} \eqref{cond2}, the set $\Delta_k$ has empty interior. So, if $V_p\cap V_q \ne \emptyset$, then $(V_p\cap V_q)\backslash \Delta_k\ne \emptyset$. From this, it follows that $\Delta_{<k}$ is connected.
\end{proof}

\begin{lemma} \label{descendlemma} Let $(Z_0,\omega_{Z_0},S\times K,\mu_{Z_0})$ be as in the discussion following Lemma \ref{vfcompletelemma}. There is a unique lattice $\Lambda\subset \partial(Q)$ in $\R^n$ so that the Hamiltonian action of $\R^n$ on $Z_0$ descends to an effective action of $G_0:=\R^n/\Lambda\cong \T^n$ on $Z_0$. 
\end{lemma}

\begin{proof} Let $V$ be a small open subset of $Z_0$ which consists of only regular points of $\mu_{Z_0}$. Take a transversal $\Upsilon$ to the foliation by $\g$ orbits in $V$; by shrinking $V$ we can assume $\Upsilon$ is connected and intersects each orbit only once. Then, by Definition \ref{toricdef} \eqref{cond6}, the action map
\[
a: (Q\ltimes \R^n \times \Upsilon \rightrightarrows \R^n \times \Upsilon) \to \mu_{Z_\bu}\n(\mu_{Z_0}(V))
\]
is a Morita equivalence. In particular, for a point $y\in V$, the stabilizer $\stab_{\R^n}(y)$ of $y$ under the Hamiltonian action of $\R^n$ on $Z_0$ is contained in $\partial(Q)$.
 
Recall that the $\R^n$-orbits through points of $\Upsilon$ are homeomorphic to $\T^n$, and that $\Upsilon$ is connected. Since the stabilizer of any point in $\Upsilon$ is contained in the discrete group $\partial(Q)$, it follows that for any $y,y'\in \Upsilon$, the stabilizer subgroups $\stab_{\R^n}(y)$ and $\stab_{\R^n}(y')$ are equal. Also, by Definition \ref{toricdef} \eqref{cond5}, for any two points $y,y'\in \mu_{Z_0}\n(p)$ in the same fiber of $\mu_{Z_0}$, the stabilizer subgroups $\stab_{\R^n}(y)$ and $ \stab_{\R^n}(y')$ are equal. So, for a connected open subset $A\subset \Delta$ consisting of only regular values, there is a unique lattice $\Lambda\subset \partial(Q)\subset \R^n$ so that $\Lambda=\stab_{\R^n}(y)$ for all $y\in \mu_{Z_0}\n(A)$. So $\R^n/\Lambda$ acts freely on $\mu_{Z_0}\n(A)$. 
Since $Z_0$ is connected, by Lemma \ref{connected} the set of regular values of $\mu_{Z_0}$ is open and connected in $\Delta$. Definition \ref{toricdef} \eqref{cond2} implies that the regular values $\mu_{Z_0}$ are dense in $\mu_{Z_0}(Z_0)$, so the action of $\R^n$ descends to an effective Hamiltonian action of $G_0:= \R^n/\Lambda \cong \T^{k}$ on $Z_0$. 
\end{proof}

By Lemma \ref{descendlemma}, we have a lattice $\Lambda\subset \partial(Q)$ and a Hamiltonian action of $G_0= \R^n/\Lambda$ on $(Z_0,\omega_{Z_0})$ with moment map $\mu_{Z_0}$. The map $Z_0\to X_0 \to \X$ intertwines the moment maps, and so the diagram 2-commutes:
\[
\begin{tikzcd}
G_0 \times Z_0 \arrow[r] \arrow[d] & Z_0 \arrow[d] \\
\G \times \X \arrow[r] & \X.
\end{tikzcd}
\]
Again applying Theorem \ref{appAthm}, we have a strict Hamiltonian action of $G_\bu= (\ker\partial\times \partial(Q)/\Lambda)\ltimes G_0$ on $Z_\bu$. The tuple $(Z_\bu,\omega_{Z_\bu},G_\bu,\mu_{Z_\bu})$ is the desired Hamiltonian groupoid presenting $\bmu\n(U)$. Notice that the definition of $G_1$ uses our fixed splitting $Q=\ker\partial \times \partial(Q)$. 
\vspace{2em}
\end{proof}

With Proposition \ref{initiallocal} in hand, we now turn our attention to the structure of $Z_0$ as a Hamiltonian $G_0$-manifold.

Let $(M,\omega,G,\mu)$ be a Hamiltonian $G\cong \T$-manifold, and $\ca{O}=G\cdot x$ an orbit in $\mu\n(p)$. Let $V=T_x\ca{O}^{\omega}/T_x\ca{O}$ be the symplectic slice at $x$, then $(V,\omega_V)$ is a symplectic vector space with Hamiltonian $G_x=\stab_G(x)$-action, with moment map
\[
\mu_V: V\to \mathfrak{g}_x^*,\quad \langle \mu_V(v),\xi \rangle = \frac{1}{2}\omega_V(\xi_V(v),v), \quad \text{ for } v\in V,~\xi\in \g_x.
\]
Pick a splitting $\mathfrak{g}^*=\mathfrak{g}_x^*\oplus \ann(\g_x)$, and consider the associated bundle 
\[
E :=  (G\times^{G_x} V) \times \ann(\g_x) \quad \Big(\cong G\times^{G_x}(\ann(\g_x)\times V) \Big)
\]
over $G$. There is a symplectic form $\omega_E$ on $E$ (which depends on the choice of splitting $\mathfrak{g}^*=\mathfrak{g}_x^*\oplus \ann(\g_x)$), and a natural $G$-action on $E$. This $G$-action is Hamiltonian, with moment map $\mu_{E}\colon E \to \g^*$ given by
\[
\mu_E([g,v],q) = q+p+\mu_V(v); \qquad [g,v]\in G\times^{G_x} V,~q\in \ann(\g_x).
\]
We call the Hamiltonian space $E=(E,\omega_E,G,\mu_E)$ the \emph{model near} $\ca{O}$.
Then, we have the following local normal form near $G\cdot x$. 

\begin{theorem}[Hamiltonian Slice Theorem]
\label{HamSlice}
Let $(M,\omega,G,\mu)$ be a Hamiltonian $G\cong \T$-manifold, and $\ca{O}=G\cdot x$ an orbit in $\mu\n(p)$. There exists a  isomorphism (which depends on the splitting of $\g^*$) of Hamiltonian $G$-manifolds between a neighborhood of $\ca{O}$ in $M$ and the zero section of the model $(E,\omega_E,G, \mu_E)$ near $\ca{O}$. \end{theorem}

\begin{remark}  \label{sliceremark}
 Let $(M,\omega,G,\mu)$ be as in the previous theorem. When $2\dim G =\dim M$ and the action of $G$ on $M$ is effective, then the model $(E,\omega_E,G,\mu_E)$ is well understood; see for instance Remark 3.7 of \cite{lerman-tolman;hamiltonian-torus-actions-symplectic-orbifolds}. In particular, the action of $G$ on the regular locus of $\mu_E$ is free and the fibers of $\mu_E$ are connected. The moment map image $\mu_E(E)$ is a convex polyhedral cone in $\g^*$, which is rational with respect to the cocharacter lattice $\Hom(S^1,G)\subset \g$.
\end{remark}

\begin{theorem} [Local Normal Form]
\label{localnormalform}
Let $(\X,\bomega,\G,\bmu)$ be a compact toric symplectic stack. Fix a quasilattice $\partial\colon Q\to \R^n$ presenting $\G$. For $p\in \bmu(\X)$, there is an open neighborhood $U$ of $p$ and a lattice $\Lambda\subset \partial(Q)$ in $\R^n$ so that
\[
(\bmu\n(U),\bomega, \G,\bmu) \simeq (\B E_\bu, \B \omega_{E_\bu}, \B G_\bu, \B \mu_{E_\bu}).
\]
Here, $(E_0,\omega_{E_0},G_0,\mu_{E_0})$ is a neighborhood of the zero section in the model $E$ near a $G_0=\R^n/\Lambda$ orbit in a Hamiltonian $G_0$ manifold $(Z_0,\omega_{Z_0},G_0,\mu_{Z_0})$, and $E_\bu$ is the action groupoid $(\ker\partial \times \partial(Q)/\Lambda)\ltimes E_0$. The group $G_0$ acts effectively on $E_0$, and $G_\bu= (\ker\partial\times \partial(Q)/\Lambda)\ltimes G_0$ is an \'etale presentation of $\G$.
\end{theorem}

\begin{proof}
There is neighborhood $U\subset \Lie(\G)^*$ of $p$ where we can assume 
\begin{equation}
\label{replaceZ}
(\bmu\n(U),\bomega,\G,\bmu)=( \B Z_\bu , \B\omega_{Z_\bu} , \B G_\bu, \B \mu_{Z_\bu}),
\end{equation}
 where $(Z_\bu,\omega_{Z_\bu},G_\bu,\mu_{Z_\bu})$ is a Hamiltonian $G_\bu$-groupoid as in Proposition \ref{initiallocal}. Then $G_0=\R^n/\Lambda$ acts effectively on $Z_0$.
 
 Let $x\in \mu_{Z_0}\n (p)$. By Theorem \ref{HamSlice}, there is a $G_0$-equivariant neighborhood of $G_0\cdot x$ in $Z_0$ which is isomorphic to a neighborhood $E_0$ of the zero section in the model $E$ near $G_0\cdot x$. Let us replace $Z_\bu$ with its restriction to this neighborhood of $G_0\cdot x$, and replace $U$ with the image of this neighborhood under the moment map $\mu_{Z_0}$. Then by Definition \ref{toricdef} \eqref{cond5} we still have that $Z_\bu$ presents $\bmu\n(U)$, and we still can assume \eqref{replaceZ}.

Consider the induced \'etale groupoid $E_\bu\cong Z_\bu$. There is a unique 0-symplectic form $\omega_{E_\bu}\in \Omega^2_0(E_\bu)$, action of $G_\bu$ on $E_\bu$, and moment map $\mu_{E_\bu}\colon E_\bu \to \g^*$ so that the isomorphism $E_\bu\cong Z_\bu$ extends to an isomorphism
\[
 (E_\bu,\omega_{E_\bu},G_\bu,\mu_{E_\bu}) \cong (Z_\bu,\omega_{Z_\bu},G_\bu,\mu_{Z_\bu})
\]
of Hamiltonian $G_\bu$-groupoids. 

It remains to investigate the structure of $E_1$. Let $H=\ker\partial \times \partial(Q)/\Lambda$, then $H$ acts on $E_0$ according to $h\cdot x := \partial(h)\cdot x$, for $h\in H$ and $x\in E_0$. Form the action groupoid $H\ltimes E_0\rightrightarrows E_0$ for this action, and consider the map of Lie groupoids
\[
F_\bu \colon H \ltimes E_0 \to E_\bu
\]
which is the identity on objects and on arrows is given by
\[
(h,x)\mapsto h*u(x),\qquad h\in H,~x\in E_0.
\]
It is easy to verify that $F_\bu$ is indeed a Lie groupoid morphism, and that $F_\bu$ is $G_\bu$-equivariant with respect to the natural action of $G_\bu$ on $H\ltimes E_0$. We will show $F_\bu$ is an isomorphism of Lie groupoids. It is clear that $F_1$ is a local diffeomorphism, and so we just need to show that it is a bijection. 

Let $x\in E_0$ be a regular point of $\mu_{E_0}$. Then by Remark \ref{sliceremark} the map $G_0 \to \mu_{E_0}\n(\mu_{E_0}(x))$ which is given by $g\mapsto g\cdot x$ is a diffeomorphism. From Definition \ref{toricdef} \eqref{cond6}, it follows that $H$ acts freely and transitively on $s\n(x)$. In particular, the restriction of $F_1$ to $H\times \{\text{regular points of }\mu_0\}$ is a bijection onto its image.

Assume we have $(h,x), (h',x')\in H\ltimes E_0$ so that $F_1(h,x)=F_1(h',x')$. Then, take a sequence $(h_i,x_i)\to (h,x)$, where the $x_i$ are all regular points of $\mu_{E_0}$. The existence of such a sequence is guaranteed by Definition \ref{toricdef} \eqref{cond2}. The restriction of $F_1$ to a neighborhood $W$ of $(h',x')$ is a diffeomorphism onto its image, so consider the preimage $F_1|_W\n(F_1(h_i,x_i))$ of the sequence $(h_i,x_i)$. This is a sequence in $W$ which converges to $(h',x')$. But, since $F_1$ is a bijection when restricted to $H\times \{\text{regular points of }\mu_0\}$, we have $F_1|_W\n(F_1(h_i,x_i))=(h_i,x_i)$. So, $(h,x)=(h',x')$ and $F_1$ is injective.

Now, let $\alpha\in E_1$. Take a small neighborhood $V$ of $\alpha$. Then, again by Definition \ref{toricdef} \eqref{cond2}, the set $s(V)\subset E_0$ contains a regular point $x$ of $\mu_{E_0}$. Then as before, $H$ acts freely and transitively on the fiber $s\n(x)\subset E_1$. So, by acting with an element of $h$, we can assume without loss of generality that $V$ is contained in the image of the identity bisection $u:E_0\to E_1$. Therefore $\alpha=u(y)$ for some $y\in E_0$, and so $\alpha=F_1(e,y)$. We have shown that $F_1$ is surjective.
\end{proof}

For a compact toric symplectic stack $(\X,\bomega,\G,\bmu)$, let $U$ be a small neighborhood of $p\in \bmu(\X)$. We will then we will call a Hamiltonian groupoid $(E_\bu,\omega_{E_\bu},G_\bu,\mu_{E_\bu})$ described in Theorem \ref{localnormalform} a \emph{local model for $\bmu\n(U)$}, or a \emph{local model for $\X$ over $U$ (near p)}.

By Remark \ref{sliceremark}, we have the following corollary.

\begin{corollary}
\label{locallypolyhedral}
Let $(\X,\bomega,\G,\bmu)$ be a compact toric symplectic stack, and fix a presentation for $\G$ by a quasilattice $\partial\colon Q\to \R^n$. For $p\in \bmu(\X)\subset \Lie(\G)^*$, a small neighborhood $U$ of $p$ in $\bmu(\X)$ is a neighborhood of a point in a simple convex polyhedral cone $C_p\subset \Lie(\G)^*$, which is rational with respect to a lattice $\Lambda\subset \partial(Q)$.
\end{corollary}

\section{Moment polytope}
\label{section5}

Let $(\X,\bomega,\G,\bmu)$ be a compact toric symplectic stack, and let $\Lie(\G)=\g \cong \R^n$. We want to associate to the moment map image $\bmu(\X)$ some additional data which will classify $(\X,\bomega,\G,\bmu)$ up to equivalence. As before we fix a presentation of $\G$ by a quasilattice $\partial\colon Q\to \R^n$.

Let $\Delta\subset \g^*$ be a convex polytope. Denote by $F=F(\Delta)$ the set of open faces of $\Delta$, let $F^{max}\subset F$ be the set of facets of $\Delta$, and let $V\subset F$ be the set of vertices of $\Delta$. Let 
\[
\ann(f) = \{ v\in \g\mid \langle p-q, v \rangle = 0\text{ for all } p,q\in f\}
\]
 be the set of elements of $\g$ which are normal to the face $f$. If, for every facet of $\Delta$, there is a nonzero element of $\ann(f)\cap \partial(Q)$, then $\Delta\in \g^*$ is \emph{quasirational}.
 
\begin{theorem} \label{convexitythm}
The moment map image $\bmu(\X)$ of a toric stack $(\X,\bomega,\G,\bmu)$ is a quasirational convex polytope.\end{theorem}

\begin{proof}
The assignment of $C_p$ to $p\in \bmu(\X)$ in Corollary \ref{locallypolyhedral} is local convexity data in the sense of Hilgert, Neeb, and Plank. By Theorem 3.10 of \cite{hilgert-neeb-plank;symplectic-convexity;compositio} and Definition \ref{toricdef} \eqref{cond4}, the moment image $\bmu(\X)$ is a convex polytope. Quasirationality follows from Corollary \ref{locallypolyhedral}.
\end{proof}

\begin{definition}
\label{stackypoly}
A \emph{decorated stacky moment polytope} is a triple $(\Delta, \G, \{\Lambda_f\}_{f\in F})$, where $\G$ is a stacky torus with fixed presentation $\partial\colon Q\to \R^n$ by a quasilattice, $\Delta$ is a simple $\partial(Q)$-quasirational polytope, and $\Lambda_f$ is a free subgroup of $\partial(Q)$ for each $f\in F$. A decorated stacky moment polytope must satisfy the following conditions:
\begin{enumerate}
\item Any $\Z$-basis of $\Lambda_f$ forms an $\R$-basis for $\ann(f)$.
\item If $f_1,f_2, \dots,f_s$ are open facets of $\Delta$, and $f$ is a face of codimension $s$ whose closure can be written as $\bar{f}=\bar{f}_1\cap\cdots \cap \bar{f}_s$, then $\Lambda_f = \Lambda_{f_1}+\cdots +\Lambda_{f_s}$. 
\end{enumerate}
An \emph{isomorphism of decorated stacky moment polytopes}
\[
\stackmorphism{\Phi}\colon (\Delta\subset \Lie(\G)^*,\G,\{\Lambda_f\}_{f\in F})\cong (\Delta'\subset \Lie(\G')^*,\G', \{\Lambda_f\}_{f\in F'})
\]
is an equivalence $\stackmorphism{\Phi}\colon \G \to \G'$ of stacky tori, so that $\Lie(\stackmorphism{\Phi})^*(\Delta') = \Delta+ c$ for some $c\in \Lie(\G)^*$. We also require that the linear isomorphism $\Lie(\stackmorphism{\Phi})\colon \Lie(\G)\cong \Lie(\G')$ induces isomorphisms of the subgroups $\Lambda_{f'}\subset\Lie(\G')$ and $\Lambda_{\Lie(\stackmorphism{\Phi})^*(f')}\subset\Lie(\G)$, for all $f'\in F'$. 
\end{definition}

As an immediate consequence of the definition, a decorated stacky moment polytope $(\Delta,\G, \{\Lambda_f\}_{f\in F})$ can be recovered from the data $(\Delta,\G,\{\Lambda_f\}_{f\in F^{max}})$, where $F^{max}\subset F$ is the set of facets of $\Delta$. It can also be recovered from the data $(\Delta,\G,\{\Lambda_v\}_{v\in V})$, where $V\subset F$ is the set of vertices of $\Delta$.

Now, let $(\X,\bomega,\G,\bmu)$ be a compact toric symplectic stack, and let $\Delta=\bmu(\X)$. For a face $f\in F(\Delta)$, let $p\in f$. 
Let $U_p\subset \Delta$ be a small neighborhood of $p$, and let $(E_\bu,\omega_{E_\bu},G_\bu,\mu_{E_\bu})$ be a Hamiltonian groupoid giving a local model over $U_p$, as in Theorem \ref{localnormalform}. Then, define 
\begin{equation} \label{polytopelabel}
\Lambda_f:=\Lambda \cap  \ann(f), \text{ where }G_0=\R^n/\Lambda.
\end{equation}
 The collection $(\bmu(\X),\G,\{\Lambda_f\}_{f\in F})$ is the \emph{decorated stacky moment polytope of $\X$}. 

\begin{proposition}
\label{polyprop}
Let $(\X,\bomega,\G,\bmu)$ be a toric stack. Fix a presentation of $\G$ by a quasilattice $\partial\colon Q\to \R^n$.
\begin{enumerate}
\item \label{poly1} The decorated stacky moment polytope of $(\X,\bomega,\G,\bmu)$ is well defined.
\item \label{poly2} The decorated stacky moment polytope $(\Delta,\G,\{\Lambda_f\}_{f\in F})$ of $\X$ satisfies the conditions of Definition \ref{stackypoly}.
\item \label{poly3} If $(\Phi_\X,\Phi_\G)\colon (\X,\bomega,\G,\bmu) \cong (\X',\bomega',\G',\bmu')$ is an equivalence of toric stacks, then $\Phi_\G$ induces an isomorphism of their decorated stacky moment polytopes.
\end{enumerate}
\end{proposition}

\begin{proof}
For \eqref{poly1}, consider an open facet $f\in \bmu(\X)$, and let $p \in f$. We first show that $\Lambda_f$ does not depend on the choices made in the construction of the local model $(E_\bu,\omega_{E_\bu},G_\bu,\mu_{E_\bu})$ for $\bmu\n(U_p)$ in Theorem \ref{localnormalform}. 

Let $S\cong \R^s\subset \R^n$ be the connected subgroup of $\R^n$ with Lie algebra $\ann(f)$. Then $\partial_S \colon \partial\n(S)\to S$ is a crossed module which presents a Lie group stack $\stack{S}$. There is an action of $\stack{S}$ on $\X$, which comes from the inclusion $S\hookrightarrow \R^n$.

Let $x\colon \star \to \bmu\n(p)$ be a categorical point of the stack $\bmu\n(p)$. Then $\X|_x = \star \times_{\bmu\n(p)} \bmu\n(p)$ is a differentiable stack which is presented by a Lie groupoid $H\rightrightarrows \star$. This groupoid can be formed explicitly as follows. For a local model $(E_\bu,\omega_{E_\bu},G_\bu,\mu_{E_\bu})$ of $\X$ over a neighborhood $U_p$ of $p$, lift $x$ to a point $x\in E_0$. Then the restriction of $E_\bu$ to the $G_0$-orbit of $x$ is isomorphic to the Lie groupoid
\begin{equation}
\label{atlasofp}
(\ker\partial \times \partial(Q)/\Lambda)\ltimes ((\R^n/\Lambda) \cdot x) \rightrightarrows (\R^n/\Lambda)\cdot x.
\end{equation} 
This is a presentation of $\bmu\n(p)$. Note that the orbit $(\R^n/\Lambda) \cdot x$ is homeomorphic to a torus of dimension $\dim f$.

Now, restrict $E_\bu$ to a (discrete) Lie groupoid over $x$. The isomorphism type of $H$ does not depend on the choice of $x$ or the local model $E_\bu$. From \eqref{atlasofp}, there is an isomorphism
\begin{equation}
\label{Hiso}
H \cong \ker\partial \times \frac{\partial(Q)\cap \ann(f)}{\Lambda\cap \ann(f)}.
\end{equation}
Since $\ann(f)$ is the Lie algebra of $S$, the action of $S$ fixes all points of $\mu_{E_0}\n(p)$.
Viewing $H\rightrightarrows \star$ as a subgroupoid of $E_\bu$, we see the crossed module $\partial_S \colon \partial\n(S)\to S$ acts strictly on $H\rightrightarrows \star$. If $x'$ is another categorical point of $\bmu\n(p)$, then the isomorphism $H\cong H'$ is $\partial\n(S)$-equivariant.

In fact, the action of $\partial\n(S)$ on $H$ is the \emph{unique} presentation of the $\stack{S}$-action on $\X|_x$, by a strict action of $\partial_S \colon \partial\n(S) \to S$ on a Lie groupoid $H\rightrightarrows \star$ whose space of objects is a single point. To show uniqueness, it is enough to show there are no other naturally isomorphic strict actions of $\partial_S \colon \partial\n(S) \to S$ on $H\rightrightarrows \star$; this follows from the fact that $H$ is discrete and abelian, and that $S\cong \R^s$ is connected. 

Therefore, if 
\[
a\colon \partial\n(S) \to H;\qquad \alpha\mapsto \alpha * u(pt)
\] is the action of $\partial\n(S)$ on the group unit $u(pt)\in H$, and $E_\bu$ is a local model of $\X$ over a neighborhood $U_p$ of $p$, then from \eqref{Hiso} we see that the lattice $\Lambda_f=\Lambda\cap \ann(f)$ defined for $E_\bu$ coincides with $\partial(\ker a)$. From the abstract description, it follows that $\Lambda_f$ does not depend on the choice of local model $(E_\bu,\omega_{E_\bu},G_\bu,\mu_{E_\bu})$.

Now, let $p'\in f$ be another point in the face $f$. Without loss of generality, we may assume that there is a local model $(E_\bu,\omega_{E_\bu},G_\bu,\mu_{E_\bu})$ of $\X$, over a neighborhood $U_p$ of $p$, and that $p'$ is contained in $U_p$. Then $E_\bu$ is isomorphic, as a Hamiltonian $G_\bu$ groupoid, to a local model over $U_p$, where now we consider $U$ as a neighborhood of $p'$. In either case the lattice $\Lambda$ is the same, and so the definition of $\Lambda_f$ really does not depend on the choice of $p\in f$.

To prove \eqref{poly2}, let $p\in f$ and consider a local model $(E_\bu,\omega_{E_\bu},G_\bu,\mu_{E_\bu})$ over an open neighborhood $U_p$ of $p$. Write $G_0=\R^n/\Lambda$. It is enough to prove that
\[
\Lambda\cap \ann(f) = (\Lambda\cap \ann(f_1)) +\cdots +(\Lambda \cap \ann(f_s)).
\]
Because $(E_0,\omega_{E_0},G_0,\mu_{E_0})$ is isomorphic to a piece of a compact toric symplectic manifold, this follows from Lemma 2.2 of \cite{delzant;hamiltoniens-periodiques}.

The statement \eqref{poly3} follows immediately from \cite[Proposition~7.4.1]{hoffman-sjamaar;hamiltonian-stack}, and from the local normal form Theorem \ref{localnormalform}.
\end{proof}

\section{Classification}
\label{section6}

We now state our main theorem. Details of its proof constitute the remainder of this article.

\begin{theorem}
\label{maintheorem}
The assignment of the decorated stacky moment polytope $(\bmu(\X),\G,\{\Lambda_f\}_{f\in F})$ to a compact toric symplectic stack $(\X,\bomega,\G,\bmu)$ defines a bijection, between the set of isomorphism classes of decorated stacky moment polytopes on one hand, and the set of equivalence classes of toric stacks, on the other.
\end{theorem}

\begin{proof}
Given a stacky moment polytope $(\Delta,\G,\{\Lambda_f\}_{f\in F})$, the existence of a toric stack $(\X,\bomega,\G,\bmu)$ with $(\bmu(\X),\G,\{\Lambda_f\}_{f\in F})=(\Delta,\G,\{\Lambda_f\}_{f\in F})$ is a corollary of Theorem \ref{deformations}, by taking the identity deformation of $\Delta$. On the other hand, the uniqueness of $\X$ up to equivalence is Theorem \ref{globaluniqueness}.
\end{proof}

\begin{remark} \label{toricorbifold}
In the context of the previous theorem, let us assume that $\G=\R^n/\Z^n$ is a torus and that $\X$ is an effective orbifold (an effective proper \'etale stack). Then $(\X,\bomega,\G,\bmu)$ is an (effective) compact toric symplectic orbifold, as classified by Lerman and Tolman in \cite{lerman-tolman;hamiltonian-torus-actions-symplectic-orbifolds}. One can recover the positive integer labels on facets of $\bmu(\X)$ described by Lerman and Tolman as follows. For each facet $f\in F^{max}$, the lattice $\Lambda_f$ is a subgroup of $Q=\partial(Q)= \Z^n\subset\g\cong \R^n$. A generator of $\Lambda_f$ is a positive integer multiple of a unique generator of $\ann(f)\cap \Z^n$. This positive integer is the label of $f$ in Theorem 6.4 of \cite{lerman-tolman;hamiltonian-torus-actions-symplectic-orbifolds}.
\end{remark}

\section{Local Uniqueness}
\label{section7}

In this section we show that a compact toric symplectic stack is determined locally by its decorated stacky moment polytope. This follows quickly from Theorem \ref{localnormalform} and the next theorem, which is a consequence of Lemma A.1 of \cite{lerman-tolman;hamiltonian-torus-actions-symplectic-orbifolds} and the Hamiltonian slice theorem (Theorem \ref{HamSlice}).

\begin{theorem}[Local uniqueness for toric manifolds] \label{uniquemodels} 
Let $(M,\omega,\T,\mu)$ and $(M',\omega',\T,\mu')$ be symplectic manifolds of dimension $2n$, with an effective Hamiltonian torus action of $\T\cong \R^n/\Z^n$. Let $x\in M$ and $x'\in M'$, and assume that there is some affine transformation $g\in Gl_n(\Z)\ltimes \R^n$ taking $\mu(x)$ to $\mu'(x')$, which is an isomorphism of the moment polytopes $\mu(M)$ and $\mu'(M')$ near $x$. Then, there is an isomorphism of Hamiltonian $\T$-manifolds, from a $\T$-equivariant neighborhood of $\T\cdot x$ to a $\T$-equivariant neighborhood of $\T\cdot x'$.
\end{theorem}

\begin{theorem}[Local uniqueness for toric stacks] \label{localuniqueness}
Let $(\X,\bomega,\G,\bmu)$ and $(\X',\bomega',\G',\bmu')$ be compact toric symplectic stacks, and let $\Phi \colon (\Delta,\G,\{\Lambda_f\}_{f\in F})\cong (\Delta',\G',\{\Lambda_{f'}\}_{f\in F'})$ be an isomorphism of their decorated stacky moment polytopes. Then, for each $p\in \Delta$, there is a small neighborhood $U_p$ and an equivalence of Hamiltonian stacks $\bmu\n(U_p)\simeq {\bmu'}\n(\Phi(U_p))$.
\end{theorem}

\begin{proof}
Without loss of generality assume $\Phi$ is the identity map. Let $U_p$ be a neighborhood of $p$ and let $(E_\bu,\omega_{E_\bu}, G_\bu,\mu_{E_\bu})$ and $(E'_\bu,\omega_{E'_\bu}, G'_\bu,\mu_{E'_\bu})$ be local models for $\bmu\n(U_p)$ and ${\bmu'}\n(U_p)$, respectively. The idea of the proof is to show that $(E_\bu,\omega_{E_\bu}, G_\bu,\mu_{E_\bu})$ is Morita equivalent to a Hamiltonian groupoid $(\hat{E}_\bu,\omega_{\hat{E}_\bu},G'_\bu, \mu_{\hat{E}_\bu})$, and then apply Theorem \ref{uniquemodels}.

Let $f$ be the open face of $\Delta$ containing $p$, and let $S\cong \R^s$ be the subgroup of $\R^n$ with Lie algebra $\mathfrak{s}=\ann(f)$. Then $G_0=\R^n/\Lambda$ and $G_0'=\R^n/\Lambda'$ can each be decomposed as
\[
G_0 = K/\Lambda_K \times S/\Lambda_f;\qquad G_0' = K'/\Lambda_{K'} \times S/\Lambda_f
\]
where $K,K'\subset \R^n$ are subspaces which are complementary to $S$, and where 
\[
\Lambda_K+\Lambda_f=\Lambda;\qquad \Lambda_{K'}+\Lambda_f=\Lambda'.
\]
 Then $K/\Lambda_K$ acts freely on $E_0$, and from Theorem \ref{HamSlice} we find
\begin{align*}
E_\bu \cong & (\ker \partial \times \partial(Q)/\Lambda)\ltimes (K/\Lambda_K \times \ann(\mathfrak{s}) \times V ). 
\end{align*}
Here $V$ is a symplectic vector space with an effective Hamiltonian action of $S/\Lambda_S$. Let $\tilde{G}_\bu$ be the Lie $2$-group 
\[
\tilde{G}_\bu = (\ker\partial \times \partial(Q)/ \Lambda_f) \ltimes( K \times S/\Lambda_f)
\]
and consider the Lie groupoid
\[
\tilde{E}_\bu = (\ker\partial \times \partial(Q)/\Lambda_f) \ltimes (K \times \ann(\mathfrak{s}) \times V),
\]
equipped with the symplectic structure and Hamiltonian $\tilde{G}_\bu$-action making the quotient map $(\tilde{E}_\bu,\omega_{\tilde{E}_\bu},\tilde{G}_\bu,\mu_{\tilde{G}_\bu})\to (E_\bu,\omega_{E_\bu},G_\bu,\mu_{G_\bu})$ a Morita equivalence of Hamiltonian groupoids.

Because $K'$ and $K$ are both subspaces of $\R^n$ which are complementary to $S$, there is a canonical isomorphism of $\tilde{G}_\bu$ with the Lie $2$-group 
\[
\tilde{G}_\bu ' = (\ker\partial \times \partial(Q)/ \Lambda_f) \ltimes( K' \times S/\Lambda_f).
\]
The action of $\Lambda_{K'}\subset \tilde{G}'_0$ on $\tilde{E}_0$ is free and proper, and there is a symplectic structure and Hamiltonian $G'_\bu$ action on
\[
\hat{E}_\bu = (\ker\partial \times \partial(Q)/\Lambda') \ltimes \left((K \times \ann(\mathfrak{s}) \times V)/\Lambda_{K'}\right)
\]
making the quotient map
\[
(\tilde{E}_\bu,\omega_{\tilde{E}_\bu},\tilde{G}'_\bu,\mu_{E_\bu}) \to (\hat{E}_\bu,\omega_{\hat{E}_\bu},G'_\bu, \mu_{\hat{E}_\bu})
\]
a Morita equivalence of Hamiltonian stacks.

Let $x\in \mu_{\hat{E}_0}\n(p)$ and let $x'\in \mu_{E'_0}\n(p)$. By Theorem \ref{uniquemodels}, there is an isomorphism of Hamiltonian $G_0'$-manifolds, from a $G'_0$-equivariant neighborhood of $x$ to a $G'_0$-equivariant neighborhood of $x'$. Since both $\hat{E}_\bu$ and $E'_\bu$ are action groupoids for the action of $\ker\partial \times \partial(Q)/\Lambda'$ on $\hat{E}_0$ and $E'_\bu$, respectively, this $G'_0$-equivariant isomorphism of manifolds can be lifted to an isomorphism of Hamiltonian $G'_\bu$-groupoids.
\end{proof}

\section{Global Uniqueness}
\label{section8}

This section is entirely devoted to the proof of the next theorem. It uses the ideas and results of a similar proof by Lerman, Tolman, and Woodward of Proposition 7.3 and Theorem 7.4 in  \cite{lerman-tolman;hamiltonian-torus-actions-symplectic-orbifolds}, which is in the context of toric symplectic orbifolds. To avoid dealing with the theory of exact sequences of stacks (over a polytope) and their cohomology, we do not fully generalize their argument to the present context. We employ instead a more direct approach. 

\begin{theorem}
\label{globaluniqueness}
Let $(\X,\bomega,\G,\bmu)$ and $(\X',\bomega',\G',\bmu')$ be toric stacks, and assume that there is an isomorphism $\Phi_\G\colon (\bmu(\X),\G,\{\Lambda_f\}_{f\in F}) \to (\bmu'(\X'),\G',\{\Lambda_f\}_{f\in F})$ of their stacky moment polytopes. Then, there is an equivalence 
\[(\Phi_\X,\Phi_\G)\colon (\X,\bomega,\G,\bmu)\to (\X',\bomega',\G',\bmu')\]
of toric stacks.
\end{theorem}

\begin{proof}

The structure of the proof is as follows. We first prove three lemmas. Lemma \ref{abstractnonsense} gives a description of certain atlases of $\X$ and $\X'$. In Lemma \ref{automorphisms}, we show that local automorphisms of symplectic toric stacks come from local automorphisms of the Hamiltonian groupoids constructed in Lemma \ref{abstractnonsense}. Finally, in Lemma \ref{removekernel} we reduce the situation to when $\G$ has trivial isotropy groups. After this setup, we give an argument based on \v{C}ech cohomology to show how to build a global equivalence $\X\simeq \X'$ from a collection of local equivalences.

Without loss of generality assume $\G'=\G$, and $\Phi_\G=\id$. Write $\bmu(\X)=\bmu'(\X')=\Delta$. Let $p_1,\dots, p_m$ be a collection of points in $\Delta$ with open neighborhoods $U_{p_1},\dots,U_{p_m}$, so that $\{U_{p_i}\}_{i=1,\dots,m}$ is an open cover of $\Delta$, and so that $\bmu\n(U_{p_i})\simeq{\bmu'}\n(U_{p_i})$ is presented by a local model 
\begin{equation} \label{herearethelocalmodels}
(E_\bu^i, \omega_{E_\bu^i},G_\bu^i, \mu_{E_\bu^i}).
\end{equation}
 Here, $G_0^i = \R^n/\Lambda^i$ is a compact torus as in Theorem \ref{localnormalform}. This is possible by Theorem \ref{localuniqueness}. We can assume the $U_{p_i}$'s are convex.

Consider the atlases
\[
A\colon \coprod_i E_0^i \to \X; \qquad A'\colon \coprod_i E_0^i \to \X'.
\]
Let $X_0 = X_0' = \coprod_i E_0^i$, let $X_1 = \coprod_{i,j} E_0^i\times_{A,\X,A} E_0^j$, and let $X_1'= \coprod_{i,j} E_0^i\times_{A',\X',A'} E_0^j$. Let $G_0=\R^n$ and let $\partial\colon Q\to G_0$ be a quasilattice presenting $\G$. Denote $G_\bu = Q\ltimes \R^n = (\ker\partial \times \partial(Q))\ltimes \R^n$.

By the following lemma, the Hamiltonian groupoids
\[
(X_\bu, \omega_{X_\bu},G_\bu , \mu_{X_\bu}),\quad (X'_\bu, \omega_{X'_\bu},G_\bu , \mu_{X'_\bu})
\]
present $\X$ and $\X'$, respectively. The natural inclusions of $E_\bu^i$ into $X_\bu$ and $X_\bu'$ preserve the symplectic structure and the Hamiltonian $G_\bu = Q\ltimes \R^n$ action.

\begin{lemma}\label{abstractnonsense}
Let $\X$ be a differentiable stack, and let
\[
\coprod_{i=1,\dots, m} E_0^i \to \X,~i=1,\dots, m
\]
be an \'etale atlas of $\X$. Choose smooth manifolds which are equivalent to the stacks $E_0^i\times_{\X} E_0^j$ and $X_0\times_\X X_0$. Then, $\X\simeq \B X_\bu$, where \[
X_0 = \coprod_i E_0^i,\quad X_1 = \coprod_{i,j} E_0^i\times_{\X} E_0^j,
\]
and the Lie groupoid structure on $X_\bu$ comes from the canonical diffeomorphism $
X_1\cong X_0\times_\X X_0.$
\end{lemma}

\begin{proof}
We just need to describe the diffeomorphism $X_1\cong X_0\times_\X X_0$. First, notice that if $A,B$ are smooth manifolds of the same dimension, then the coproduct $A\sqcup B$ as manifolds is equivalent to the weak coproduct $A\sqcup^{st} B$ as stacks: For any stack $\stack{Y}$, there are canonical equivalences 
\[
\Hom(A\sqcup B,\stack{Y}) \simeq \stack{Y}(A\sqcup B)\simeq \stack{Y}(A)\times \stack{Y}(B)\simeq \Hom(A,\stack{Y})\times \Hom(B,\stack{Y})\simeq \Hom(A\sqcup^{st} B,\stack{Y})
\]
which are consequences of the 2-Yoneda lemma, the definition of a stack, the 2-Yoneda lemma, and the universal property of $\sqcup^{st}$, respectively. So, $A\sqcup B$ satisfies the universal property of $A\sqcup^{st} B$, and hence they are equivalent. Similarly, if $A_\bu$ and $B_\bu$ are Lie groupoids presenting smooth manifolds of the same dimension, then $\B(A_\bu \sqcup B_\bu) = \B(A_\bu)\sqcup^{st} \B(B_\bu)$. Let us then write $\sqcup$ for $\sqcup^{st}$ in what follows.

Now, if $A,B,C$ are manifolds, $\stack{Y}$ is a differentiable stack, and there are representable morphisms $A,B,C \to \stack{Y}$, then there is an equivalence
\[
A\times_{\stack{Y}} (B\sqcup C) \simeq A\times_{\stack{Y}} B \sqcup A \times_{\stack{Y}} C,
\]
which is canonical up to natural isomorphism. To verify this, pick a presentation $Y_\bu$ of $\stack{Y}$, and Lie groupoid presentations $A_\bu,B_\bu,C_\bu$ of $A,B,C$, respectively, so that the morphisms of stacks $A,B,C\to \stack{Y}$ come from morphisms of Lie groupoids $A_\bu ,B_\bu ,C_\bu \to Y_\bu$. Then, compute explicitly using weak pullbacks of Lie groupoids.

It follows that there is an equivalence of stacks $X_1\simeq X_0\times_\X X_0$. But $X_1$ and $X_0\times_\X X_0$ are manifolds, so this equivalence is naturally isomorphic to a unique diffeomorphism.
\end{proof}

\begin{lemma}\label{automorphisms}
Fix one of the local models $(E_\bu= E_\bu^i, \omega_{E_\bu}, G_\bu^i, \mu_{E_\bu})$ from \eqref{herearethelocalmodels}. Let $U=U_{p_i}$, let $V\subset U$ be contractible and open in $\Delta$, and let $Y= \mu_0\n(V)\subset E_0$. If $\stackmorphism{\phi}\colon \bmu\n(V) \to \bmu\n(V)$ is an equivalence of Hamiltonian $\G$-stacks, then there is an automorphism $\phi_0\colon Y\to Y$ of Hamiltonian $G_0^i$-manifolds so that the diagram 2-commutes:
\[
\begin{tikzcd}
Y \arrow[r, "\phi_0"] \arrow[d, "A"]
& Y \arrow[d, "A" ] \\
\bmu\n(V) \arrow[r, "\stackmorphism{\phi}"]
& \bmu\n(V)
\end{tikzcd}
\]
Here, $A\colon Y\to \bmu\n(V)$ is the atlas $A\colon X_0 \to \bmu\n(U)$ restricted to $Y$.
\end{lemma}

\begin{proof}
Consider the pullback diagram
\[
\begin{tikzcd}
Z \arrow[rr, "p_2"] \arrow[d, "p_1"]
&& Y \arrow[d , "A"] \\
Y\arrow[r, "A"] & \bmu\n(V) \arrow[r, "\stackmorphism{\phi}"]
& \bmu\n(V)
\end{tikzcd}
\]
The manifold $Z= Y\times_{\stackmorphism{\phi}\circ A, \bmu\n(V),A} Y$ can be interpreted as a submanifold of the space of arrows of an \'etale Lie groupoid $W_\bu$ presenting $\bmu\n(V)$, where $W_0=Y\sqcup Y$ has the atlas $ (\stackmorphism{\phi} \circ A)\sqcup A \colon Y\sqcup Y \to \bmu\n(V)$. More precisely, $Z$ is the submanifold of $W_1$ consisting of arrows with source in the first copy of $Y$ and target in the second copy of $Y$. Then $p_1$ is the restriction to $Z$ of the source map of $W_\bu$, and $p_2$ is the restriction to $Z$ of the target map.

From Theorem \ref{appAthm}, there is a symplectic structure and Hamiltonian $G_1^i$-action on $Z$, making $p_1$ and $p_2$ maps of Hamiltonian $G_0^i$-manifolds. (Here we use the identity bisection $u\colon G_0^i\to G_1^i$ to consider $G_0^i$ as a subgroup of $G_1^i$). It suffices to show that there is a global section $\sigma$ of $p_1$; the map $\sigma$ preserves the moment map and symplectic form, and hence  is automatically a map of Hamiltonain $G_0^i$ manifolds. We will then set $\phi_0= p_2\circ \sigma$.

Let $H= \ker\partial\times \partial(Q)/\Lambda^i$. We will show that $p_1\colon Z\to Y$ is a trivial $H$-principal bundle. First, the action of $H \subset G_1^i$ on $W_1$ is parallel to fibers of $s\colon W_1\to W_0$. In particular, $H$ acts on the fibers of $p_1$. Now, let $\alpha, \alpha' \in Z$, so that $p_1(\alpha) = p_1(\alpha')$. Then, there is a unique arrow $\beta \in Y\times_{A,\bmu\n(V),A} Y\subset W_1$ so that $\beta\circ \alpha = \alpha'$. From the description of $E_\bu$ in Theorem \ref{localnormalform}, we know $Y\times_{A,\bmu\n(V),A} Y\cong H \ltimes Y$. So, recalling the notation for the action of a crossed module from \eqref{xmodact}, there is a unique $h\in H$ so that $h * u(t(\alpha)) = \beta$. But, by~\eqref{equation;compatcross}, we find
\[
\alpha' = \beta\circ \alpha = (h * u(t(\alpha)))\circ \alpha = h * \alpha.
\] So, the action of $H$ on fibers of $p_1$ is free and transitive. Therefore, $p_1$ is a principal $H$-bundle.

Assume that for each $x \in Y$, there is a connected $G_0^i$-equivariant neighborhood $S_x$ of $x$ and an (automatically $G_0^i$-equivariant) local section $\sigma_x \colon S_x \to Z$ of $p_1$. Then the transition maps between the local trivializations given by the $\sigma_x$ are the same as the transition maps for a principal $H$-bundle over $Y/G_0^i=\mu_Y(Y) = V$. But $V$ is contractible, so any principal bundle will be trivial. So, in this case there is a global section $\sigma$ of $p_1$.

It suffices then to show that for each $x \in Y$, there is a $G_0^i$-equivariant neighborhood $S_x$ of $x$ and a local section $\sigma_x \colon S_x \to Z$ of $p_1$. Fix $x$, and let $\sigma \colon T\to Z$ be a local section of $p_1$ on a small open neighborhood $T$ of $x$. 
Since $G_0^i$ is compact and connected, after possibly shrinking $T$ we can assume that for any $y\in T$, the set $T\cap (G^i_0\cdot y)$ is connected. 
We will show that $g\cdot \sigma(y) = \sigma(g\cdot y)$ for any $y\in T$ and any $g\in G_0^i$ which satisfies $g\cdot y\in T$. Then, we can define $S_x = G_0^i\cdot T$, and the section
\[
\sigma_x\colon S_x\to Z; \quad g\cdot y\mapsto g \cdot \sigma(y)
\]
of $p_1$ is well defined.

First, assume $g\cdot y = y$. Then since the stabilizer of $y$ in $G_0^i$ is connected, there is some $v \in \g_y$ so that $g=\exp(v)$, where $\g_y$ is the kernel of the linear map $\g \to T_y Y$ given by the action of $\g$. Since $\sigma$ is a local section of $p_1$, it intertwines the moment maps and hence the commutes with the $\g$ action. So, $v\in \g_{\sigma(y)}$ and consequently $g\cdot \sigma(y) =\sigma(y)$. 

Now, consider any $g\in G_0^i$ and $y\in T$, such that $g\cdot y \in T$. Then, since $T\cap (G^i_0\cdot y)$ is connected, there is some $\tilde{g}\in G_0$ near the identity so that $\tilde{g}\cdot y =g\cdot y$, and so that $\sigma(\tilde{g}\cdot z) = \tilde{g}\cdot  \sigma(z)$ as long as $z\in T$ and $\tilde{g}\cdot z\in T$. Then,
\[
\sigma(\tilde{g}\n g\cdot y) = \tilde{g}\n \cdot \sigma(g\cdot y).
\]
And, by the previous paragraph,
\[
\sigma(\tilde{g}\n g\cdot y) =\tilde{g}\n g \cdot \sigma(y)
\] 
since $\tilde{g}\n g\cdot y= y$. So, $\sigma(g\cdot y)=g\cdot  \sigma(y)$, as desired.
\end{proof}

We now turn our attention to the isotropy groups of $X_\bu$. It is easy to verify that, for $\alpha,\beta,\alpha\circ\beta \in X_1$ and $q,r\in \ker\partial$, 
\begin{equation}
\label{actioncommutes}
(q*\alpha)\circ (r*\beta) = (qr)*(\alpha\circ \beta);\qquad (q*\alpha)\n=q\n*\alpha\n.
\end{equation}
So, the Lie groupoid structure of $X_\bu$ descends to a Lie groupoid structure on $X_\bu^{str}:= X_1/\ker\partial\rightrightarrows X_0$, which is $X_\bu$ with some isotropy groups removed (``strictified''). Then $X_\bu^{str}$ has the 0-symplectic form $(\omega_{X^{str}_0} = \omega_{X_0}, \omega_{X^{str}_1}= s^*\omega_{X^{str}_0})$. The Hamiltonian action of $G_\bu$ on $X_\bu$ descends to a Hamiltonian action of $G^{str}_\bu:=G_1/\ker\partial\rightrightarrows G_0$ on $X_\bu/\ker\partial$, with moment map $(\mu_{X^{str}_0} = \mu_{X_0}, \mu_{X^{str}_1}= s^*\mu_{X^{str}_0})$.
Let 
\[
(\X^{str}=\B X_\bu^{str},\bomega^{str},\G^{str}, \bmu^{str})
\] 
be the Hamiltonian stack presented by the strictification $X_\bu^{str}$ of $X_\bu$.

We know that $G_\bu$ is a trivial $(\ker\partial\rightrightarrows \star)$-extension of $G_\bu^{str}$. And $X_\bu$ is also a $(\ker\partial\rightrightarrows \star)$-extension of $X_\bu^{str}$; we will show it is a trivial extension as well. In other words, we will show that the natural map $\stackmorphism{\pi}\colon \X\to \X^{str}$ is a trivial $\ker\partial$-gerbe over $\X^{str}$. 

Using our fixed splitting $G_\bu = (\ker\partial \times \partial(Q))\ltimes \R^n$, it will be helpful to identify $G_\bu^{str}$ with a (non-full) Lie 2-subgroup of $G_\bu$. In particular one can think of $\X$ as a Hamiltonian $\G^{str}$-stack.

\begin{lemma}
\label{removekernel}
Let $\tilde{X}_0=X_0$ and let $\tilde{X}_1= \ker\partial \times (X_1/\ker\partial)$. Equip $\tilde{X}_\bu$ with the  symplectic structure $\omega_{\tilde{X}_\bu}$ coming from $\tilde{X}_0=X_0$. The action of $G_\bu \cong (\ker\partial \times \partial(Q))\ltimes \R^n$ on $\tilde{X}_\bu$, which is given on arrows as
\begin{align*}
\Big( (\ker \partial \times \partial(Q)) \ltimes \R^n\Big) \times \Big( \ker\partial \times (X_1/\ker\partial) \Big)& \to \ker\partial \times (X_1/\ker\partial) \\
((q,h, g),(r, \alpha)) & \mapsto (qr, h*(g\cdot \alpha))
\end{align*}
is Hamiltonian. 
Then, there is an isomorphism of Hamiltonian $G_\bu$-groupoids
\[
(X_\bu,\omega_{X_\bu}, G_\bu, \mu_{X_\bu}) \cong (\tilde{X}_\bu,\omega_{\tilde{X}_\bu},G_\bu,\mu_{\tilde{X}_\bu}).
\]
which is the identity on the manifold of objects $X_0=\tilde{X}_0$.
\end{lemma}

\begin{proof}
It is obvious that the action described is Hamiltonian, with moment map $\mu_{\tilde{X}_0}=\mu_{X_0}$.  Let $V$ be the interior of $\Delta$, then $V$ consists of the regular values of $\bmu$. We will first show that $\stackmorphism{\pi}|_{\bmu\n(V)}\colon \bmu\n(V)\to (\bmu^{str})\n(V)$ is a trivial principal $\B(\ker\partial \rightrightarrows \star)$-bundle (that is, a trivial $\ker\partial$ gerbe). 

There is a presentation of $(\bmu)\n(V)$ by an \'etale Hamiltonian groupoid $(Z_\bu,\omega_{Z_\bu},G_\bu,\mu_{Z_\bu})$, so that the actions of $G_0$ and $G_1$ on $Z_0$ and $Z_1$, respectively, are free. To construct the  manifold $Z_0$, pick an \'etale presentation $\tilde{Z}_\bu$ of $(\bmu)\n(V)$. Then, pick an open cover $\{V_{q_i}\}$ of $V$, so that $V_{q_i}$ is a neighborhood of $q_i\in V$ with a section $\sigma_{q_i}\colon V_{q_i}\to \tilde{Z}_0$ of the moment map $\mu_{\tilde{Z}_0}\colon \tilde{Z}_0 \to \Lie(\G)^*$ . Finally, let $Z_0 = G_0\times \coprod_{i} V_{q_i}$. The atlas $Z_0\to (\bmu)\n(V)$ is described in \cite[Proposition~B.2]{hoffman-sjamaar;hamiltonian-stack}, and the existence of a Hamiltonian action of $G_\bu$ on $Z_\bu$ is guaranteed by Theorem \ref{appAthm}. Freeness of the $G_1$ action is guaranteed by Definition \ref{toricdef} \eqref{cond6}.

Define
\[
Z_\bu^{str}:= (Z_1/\ker\partial\rightrightarrows Z_0); \quad Z_\bu/G_\bu^{str} := (Z_1/G_1^{str} \rightrightarrows Z_0/G_0).
\]
Note that $Z_\bu^{str}$ is a presentation for $(\bmu^{str})\n(V)$. Then, the diagram is 2-cartesian (in the 2-category of Lie groupoids):
\[
\begin{tikzcd}
Z_\bu \arrow[r] \arrow[d] & Z_\bu^{str} \arrow[d] \\
Z_\bu/G_\bu^{str} \arrow[r] & V.
\end{tikzcd}
\]
Since the 2-functor $\B$ preserves weak pullbacks, the diagram is still 2-cartesian:
\begin{equation}
\label{stackcartgerbe}
\begin{tikzcd}
\bmu\n(V) \arrow[r] \arrow[d] & (\bmu^{str})\n(V) \arrow[d] \\
\B(Z_\bu/G_\bu^{str}) \arrow[r] & V.
\end{tikzcd}
\end{equation}
Then, $\B(Z_\bu/G_\bu^{str}) \to V$ is a $\ker\partial$-gerbe over $V$. Recall from \cite{giraud;cohomology-nonabeliene} that $\ker\partial$-gerbes over $V$ are classified up to equivalence by the \v{C}ech cohomology group $H^2(V, \ker\partial) $. But $V$ is contractible, so there is a section $V\to \B(Z_\bu/G_\bu^{str})$ of the bottom map in \eqref{stackcartgerbe}. 
Since the diagram \eqref{stackcartgerbe} is 2-cartesian, we then have a morphism of Hamiltonian $\G^{str}$-stacks
\[
\stackmorphism{\psi}\colon (\bmu^{str})\n(V) \to \bmu\n(V)
\]
so that $\stackmorphism{\pi}\circ \stackmorphism{\psi}$ is naturally isomorphic to the identity map of $(\bmu^{str})\n(V)$.

Passing to the presentation $X_\bu$, there is a Lie groupoid morphism $\psi_\bu$
\[
\begin{tikzcd}
X_1^{str} \arrow[d, shift right] \arrow[d,shift left] \arrow[r, "\psi_1"]
& X_1 \arrow[d, shift right] \arrow[d,shift left] \\
X_0 \arrow[r, equal, "\psi_0"]
&X_0
\end{tikzcd}
\]
so that $\pi_{X_\bu}\circ \psi_\bu\cong \id_{X_\bu^{str}}$, where $\pi_{X_\bu}\colon X_\bu\to X_\bu^{str}$ is the quotient map. Here, $\psi_1$ is defined only on the open dense subset $\mu_{X_1^{str}}\n(V)\subset X_1^{str}$. Since $X_\bu^{str}$ has trivial isotropy groups over the regular locus of $\mu_{X_0}$, it turns out that $\pi_{X_\bu}\circ \psi_\bu= \id_{X_\bu^{str}}$ where it is defined.

Now, by construction $\stackmorphism{\psi}$ preserves the action of $\G^{str}$ up to isomorphism. So the map $\psi_1$ respects the action of $G_1^{str}$ up to a natural transformation $\eta\colon G_0\times X_0 \to X_1$. However, using the fact that $G_0$ is connected and that the isotropy groups of $X_1$ are discrete, one can show that $\eta(g,x) = \eta(e,g\cdot x)$ for all $g\in G_0$ and $x\in X_0$. As a consequence, $\psi_1$ respects the action of $G_1^{str}$ on the nose.

Since removing the critical points of $\mu_{X_1^{str}}$ does not disconnect any connected open neighborhoods in $X_1^{str}$, there is a unique smooth extension of the map $\psi_1$ to all of $X_1/\ker\partial$. We then have a morphism of Hamiltonian $G_\bu^{str}$-groupoids
\[
\psi_\bu\colon X_\bu^{str}\to X_\bu
\]
which is a section of the quotient map $X_\bu \to X_\bu^{str}$. Composing the action of $\ker\partial\rightrightarrows \star$ with the morphism $\psi_\bu$ gives the morphism 
\[
a\circ(\id\times \psi_\bu) \colon  \tilde{X}_\bu = (\ker\partial\rightrightarrows \star) \times X_\bu^{str} \to X_\bu
\]
which is the desired isomorphism of Lie groupoids; it extends in the obvious way to an isomorphism of Hamiltonian $G_\bu = (\ker\partial\rightrightarrows \star)\times G^{str}_\bu$ groupoids.
\end{proof}

Due to Lemma \ref{removekernel}, to prove Theorem \ref{globaluniqueness} we will now restrict to the case where $\ker \partial =0$. We then have $Q=\partial(Q)$, and $G_1 = \partial(Q)\ltimes \R^n$.

Let $U_{i,j} = U_{p_i}\cap U_{p_j}$, and let $E^{i,j}_\bu= \mu_{E_\bu^i}\n(U_{i,j})$. Consider the equivalences $\stackmorphism{A}^i \colon \B E^i_\bu \simeq \bmu\n(U_i)$, which have the property that the atlas $A|_{E^i_0}\colon E^i_0 \to \bmu\n(U_i)$ factors as $A|_{E^i_0 } = \stackmorphism{A}^i \circ \B \iota$, where $\iota\colon E^i_0\to E^i_\bu$ is the natural inclusion. Choose weak inverses ${\stackmorphism{A}^i}\n$ of these equivalences $\stackmorphism{A}^i$. Similarly, consider the equivalences ${\stackmorphism{A}^i}'\colon \B E_\bu^i \simeq{\bmu'}\n(U_i).$
 Define $\stackmorphism{h}^i$ as the composition
\[
\begin{tikzcd}
\stackmorphism{h}^i\colon \bmu\n(U_i) \arrow[r, "{\stackmorphism{A}^i}\n"] & \B E_\bu^i \arrow[r, "{\stackmorphism{A}^i}'"] & {\bmu'}\n (U_i),
\end{tikzcd}
\]
Choose weak inverses ${\stackmorphism{h}^i}\n$ for the equivalences $\stackmorphism{h}^i$, and define 
\[
\stackmorphism{\phi^{i,j}} = {\stackmorphism{h}^i}\n \circ \stackmorphism{h}^j\colon \bmu\n(U_{i,j})\to \bmu\n(U_{i,j})
\]
Notice that the morphisms $\stackmorphism{\phi}^{i,j}$ satisfy the weak cocycle conditions
\begin{equation}
\label{stackcocycle}
\stackmorphism{\phi}^{i,j}\circ \stackmorphism{\phi}^{j,i}\cong \id;\qquad \stackmorphism{\phi}^{ij}\circ \stackmorphism{\phi}^{j,k}\circ \stackmorphism{\phi}^{k,i}\cong \id.
\end{equation}
By Lemma \ref{automorphisms}, there is an automorphism $\phi^{i,j}_0\colon E^{i,j}_0\to E^{i,j}_0$ of Hamiltonian $G_0^i$-manifolds so that the diagram 2-commutes:
\begin{equation}
\label{phidiag}
\begin{tikzcd}
E^{i,j}_0 \arrow[r, "\phi_0^{i,j}"] \arrow[d, "A"]
& E^{i,j}_0 \arrow[d, "A" ] \\
\bmu\n(U_{i,j}) \arrow[r, "\stackmorphism{\phi}^{i,j}"]
&\bmu\n(U_{i,j})
\end{tikzcd}
\end{equation}
From the proof of Proposition 7.3 of \cite{lerman-tolman;hamiltonian-torus-actions-symplectic-orbifolds}, there is a smooth function $f^{i,j} \colon U_{i,j} \to \R$ so that the time 1 flow of the Hamiltonian vector field of $f^{i,j}\circ \mu_{E_0^{i,j}}$ is $\phi^{i,j}_0$.

Let $E_\bu^{k;i,j} =\mu_{E_\bu^k}\n(U_{i,j})$, let $U_{i,j,k} = U_{p_i}\cap U_{p_j}\cap U_{p_k}$, and let
\[
\phi^{k; i,j}_\bu \colon E^{k; i,j}_\bu \to E^{k;i,j}_\bu
\]
be the time 1 flow of the Hamiltonian vector field of $f^{i,j}\circ \mu_{E_\bu ^k}|_{E_\bu^{k;i,j}}$. Note that these Hamiltonian vector fields are well defined because $E_\bu$ is \'etale and hence $\omega_{E_1}$ and $\omega_{E_0}$ are both nondegenerate. Note also that the Hamiltonian vector fields of $f^{i,j}\circ \mu_{E_k^{i,j}}$ are parallel to fibers of $\mu_{E_k^{i,j}}$, for $k=0,1$. These fibers are unions of compact tori, and so the vector fields are complete.

 Then, the diagram 2-commutes:
\[
\begin{tikzcd}
E^{k; i,j}_0 \arrow[r, "\phi_0^{k; i,j}"] \arrow[d, "A"]
& E^{k; i,j}_0 \arrow[d, "A" ] \\
\bmu\n(U_{i,j,k }) \arrow[r, "\stackmorphism{\phi}^{i,j}"]
&\bmu\n(U_{i,j,k})
\end{tikzcd}
\]
From the cocycle conditions \eqref{stackcocycle}, we see that there are natural isomorphisms of maps of Lie groupoids
\[
\phi^{i;i,j}_\bu \circ \phi^{i;j,i}_\bu \cong \id;\qquad \phi^{i;i,j}_\bu \circ \phi^{i;j,k}_\bu \circ \phi ^{i;k,i}_\bu\cong \id.
\]
Since $E^{i;i,j}_0$ is connected and $E^{i;i,j}_\bu$ has the structure of an action groupoid of the discrete group $\partial(Q)/\Lambda^i$, for all $x\in E^{i;i,j}_0$ we have
\[
(\phi^{i;i,j}_0 \circ \phi^{i;j,i}_0 )(x) = \partial(q)\cdot x
\]
for some $q\in Q$. Similarly, for all $x\in E^{i;i,j}_0$,
\[
(\phi^{i;i,j}_0 \circ \phi^{i;j,k}_0 \circ \phi ^{i;k,i}_0)(x) =\partial(r)\cdot x.
\]
for some $r\in Q$. 
From this, there are constants $c,d\in \R$ and $q,r \in Q$, so that for all $x\in U_{i,j}$ and all $y\in U_{i,j,k}$, 
\[
f^{i,j}(x)+f^{j,i}(x) = c + \langle \partial(q), x \rangle ; \qquad f^{i,j}(y) + f^{j,k}(y) + f^{k,i}(y) = d + \langle \partial(r) , y \rangle.
\]
From a sheaf-theoretic argument as in the proof of Proposition 7.3 of \cite{lerman-tolman;hamiltonian-torus-actions-symplectic-orbifolds}, there are smooth functions $f^i\colon U_i\to \R$ so that for any $x\in U_{i,j}$,
\begin{equation}\label{coboundary}
f^i(x)-f^j(x) = f^{i,j}(x) + c+ \langle v, x \rangle
\end{equation}
for some $c\in \R$ and $v\in \partial(Q)$ which depend only on the indices $i,j$.

Now, taking pullbacks of \eqref{phidiag}, the diagram 2-commutes:
\[
\begin{tikzcd}
E^{i,j}_0 \arrow[r, "\phi_0^{i,j}"] & E^{i,j}_0 \arrow[d, "A'"]  \\
E^i_0 \times_\X E^j_0  \arrow[u] \arrow[d] & \X' \\
E^{j,i}_0 \arrow[r, equal]  & E^{j,i}_0 \arrow[u, "A'"]
\end{tikzcd}
\]
Let $\phi_0^i\colon E^i_0\to E^i_0$ be the time 1 flow of the Hamiltonian vector field of $f^i\circ \mu_{E_0^i}$. Then by \eqref{coboundary}, the diagram 2-commutes:
\[
\begin{tikzcd}
E^{i,j}_0 \arrow[r, "\phi_0^{i}"] & E^{i,j}_0 \arrow[d, "A'"]  \\
E^i_0 \times_\X E^j_0  \arrow[u] \arrow[d] & \X' \\
E^{j,i}_0 \arrow[r, "\phi_0^j"]  & E^{j,i}_0 \arrow[u, "A'"]
\end{tikzcd}
\]
So, there is a map $\phi^{i\times j}_1\colon E_0^i\times_\X E_0^j\to E_0^i\times_{\X'} E_0^j$ making the following diagram commute:
\begin{equation}
\label{twosquarediag}
\begin{tikzcd}
E^{i}_0 \arrow[r, "\phi_0^{i}"] & E^{i}_0  \\
E^i_0 \times_\X E^j_0  \arrow[u] \arrow[d] \arrow[r, "\phi^{i\times j}_1"]&E^i_0\times_{\X'} E^j_0 \arrow[d] \arrow[u] \\
E^{j}_0 \arrow[r, "\phi_0^j"]  & E^{j}_0 
\end{tikzcd}
\end{equation}
By construction, any choice of $\phi^{i\times j}_1$ respects the source and target maps:
\[
s(\phi^{i\times j}_1(\alpha))=\phi^i_0(s(\alpha)); \qquad t(\phi^{i\times j}_1(\alpha))=\phi^j_0(t(\alpha)).
\]
Recall that, due to Lemma \ref{removekernel}, we have assumed $\ker \partial =0$. In this case, by Theorem \ref{localnormalform}, an arrow $\alpha\in X_1$ in the regular locus of  $\mu_{X_1}$ (or an arrow $\alpha'$ in the regular locus of $\mu_{X_1'}$) is uniquely determined by its source $s(\alpha)$ and target $t(\alpha)$. By Definition \ref{toricdef}\eqref{cond2}, the regular locus of $\mu_{X_1}$ is dense in $X_1$ and so there are \emph{unique} maps $\phi^{i\times j}_1$ which fit into the diagram \eqref{twosquarediag}. These maps together respect the groupoid structure and $G_\bu$ action on $X_\bu$. By Theorem \ref{localuniqueness}, the maps $\phi^{i\times i}_1$ are diffeomorphisms. So, the $\phi^{i\times j}_1$ assemble to an isomorphism of Hamiltonian $G_\bu$ groupoids $X_\bu\cong X_\bu'$.\end{proof}

\section{Existence and deformations of compact toric symplectic stacks}
\label{section9}

In this section we study smooth deformations of stacky moment polytopes and toric stacks, depending on a parameter $\tau\in (0,1)$. By extending Delzant's construction in \cite{delzant;hamiltoniens-periodiques} we show that any deformation of stacky moment polytopes arises from a deformation of toric stacks. In particular, by taking the constant deformation of a moment polytope one finds that any stacky moment polytope comes from a toric stack. 

Deformations of Lie groupoids were introduced and studied in a more general context by Crainic, Mestre, and Struchiner~\cite{crainicDeformationsLieGroupoids2015}.

\begin{definition}
A \emph{deformation of stacky tori} is a bundle of stacky Lie groups
\[
p_\G \colon \G\to (0,1)
\]
so that $\G^\tau:=p_\G \n(\tau)$ is a stacky torus. More precisely, $\G\to (0,1)$ is a stacky Lie groupoid in the sense of Definition 3.1 of \cite{bursztyn-noseda-zhu;principal-stacky-groupoids}, with the additional property that the source and target maps are both equal to $p_\G$.
\end{definition}

We will consider deformations of stacky tori which arise in a particular way, as follows. Fix a finitely generated abelian group $Q$ and an isomorphism $Q \cong R\times \Z^m$, where $R$ is the torsion part of $Q$. Let $\partial^\bu_i\colon (0,1)\to \R^n, i\in \{1,\dots, m\}$, be smooth paths in $\R^n$. For a given $\tau\in (0,1)$, the vectors $\partial^\tau_i$ determine a homomorphism 
\[
\partial^\tau\colon Q\cong R\times \Z^m \to \R^n;\quad (r,x_1,\dots,x_m)\mapsto x_1 \partial^\tau_1+\cdots +x_m \partial^\tau_m
\] 
and hence a stacky torus. We therefore have a deformation of stacky tori, presented by the action groupoid
\[
Q\ltimes (\R^n\times (0,1)) \rightrightarrows \R^n \times (0,1),
\]
where the action of $Q$ on $\R^n \times (0,1)$ is
\[
h\cdot (x,\tau) = (\partial^\tau(h) + x, \tau); \qquad h\in Q,~(x,\tau)\in \R^n\times (0,1).
\]
 We will denote this deformation by $\B(\partial^\bu\colon Q\to \R^n)$. Notice that the Lie algebra of the stacky torus $\B(\partial^\tau\colon Q\to \R^n)$ is constant with respect to $\tau$.

\begin{definition}
Fix a deformation of stacky tori $\G = \B(\partial^\bu\colon Q\to \R^n)$, with Lie algebra $\g=\Lie(\G^\tau)$. A \emph{deformation of stacky moment polytopes} is a 1-parameter family
\[
(\Delta^\tau,\G^\tau, \{\Lambda_{f}^\tau\}_{f\in F}, \{q_f\}_{f\in F^{max}}, \{L_f^\tau\}_{f\in F^{\max}});\quad \tau\in (0,1)
\]
so that $(\Delta^\tau,\G^\tau , \{\Lambda_{f}^\tau\}_{f\in F})$ is a decorated stacky moment polytope at any given $\tau$, and which satisfies the following conditions:
\begin{enumerate}
\item The set of $F$ labeling the faces of $\Delta^\tau$ is constant with respect to $\tau$. 
\item For each facet $f\in F^{max}$, there is a fixed $q_f\in Q$, so that so that at a given $\tau\in (0,1)$, the group $\Lambda_f^\tau$ is generated by $\lambda_f^\tau:= \partial^\tau(q_f)$.
\item For each facet $f\in F^{max}$, there is a smooth function $L_f^\bu \colon (0,1)\to \R$ so that the polytope $\Delta^\tau$ is cut out by inequalities
\[
\Delta^\tau = \{ \xi \in \g^* \mid \langle \xi, \lambda_f^\tau \rangle \le L_f^\tau, \forall f\in F^{max} \}
\]
at a given $\tau\in (0,1)$. 
\end{enumerate}
\end{definition}

By perturbing slightly the quasilattice of a stacky moment polytope, we immediately have the following.

\begin{proposition}
\label{deformprop}
Any decorated stacky moment polytope $(\Delta, \B(\partial\colon Q\to \R^n), \{\Lambda_f\}_{f\in F})$ admits a deformation to a decorated stacky moment polytope $(\Delta', \B(\partial'\colon Q\to \R^n), \{\Lambda_f'\}_{f\in F})$, so that $\Delta'$ is rational with respect to $\Z^n\subset \R^n$, and the image of $\partial'$ is $\Z^n$.\end{proposition}

\begin{example}[Figure \ref{examplefig}]
\label{nudgetriangle}
 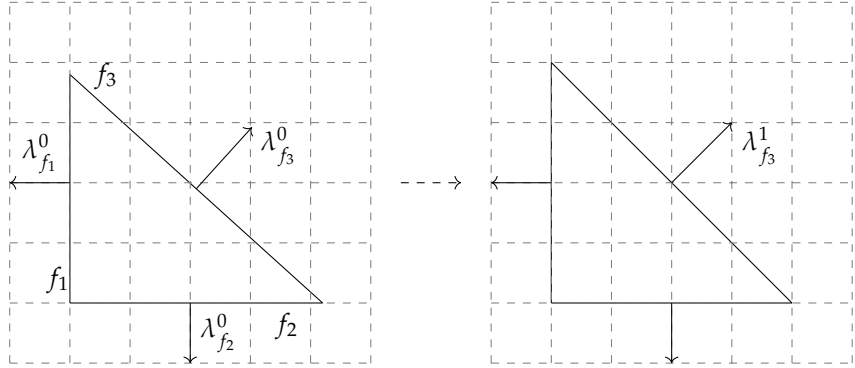
\begin{figure}[htp]\label{examplefig}
\centering
\begin{tikzpicture}[scale=1.6]	
\draw (0,0) -- (0,1.9);
\draw (-.1,0) node[align=right, above] {$f_1$};
\draw (0,1.9) -- (2.1,0);
\draw (.3,1.7) node[align=right, above] {$f_3$};
\draw (0,0) -- (2.1,0);
\draw (1.8,0) node[align=right, below] {$f_2$};
\draw[step=.5cm,gray,dashed] (-.5,-.5) grid (2.5,2.5);
\draw[->] (0,1) -- (-.5,1);
\draw (-.25,1) node[align=right, above] {$\lambda_{f_1}^0$};
\draw[->] (1,0) -- (1,-.5);
\draw (1,-.25) node[align=right, right] {$\lambda_{f_2}^0$};
\draw[->] (1.05,.95) -- (1.05+0.67091/1.45,.95+0.74153/1.45);
\draw (1.5,1.3) node[align=right, right] {$\lambda_{f_3}^0$};
\draw[dashed,->] (2.75,1) -- (3.25,1);
\begin{scope}[shift={(4,0)}]
\draw (0,0) -- (0,2);
\draw (0,2) -- (2,0);
\draw (0,0) -- (2,0);
\draw[step=.5cm,gray,dashed] (-.5,-.5) grid (2.5,2.5);
\draw[->] (0,1) -- (-.5,1);
\draw[->] (1,0) -- (1,-.5);
\draw[->] (1,1) -- (1.5,1.5);
\draw (1.5,1.3) node[align=right,right] {$\lambda_{f_3}^1$};
\end{scope}
	\end{tikzpicture}	
	\caption{Example \ref{nudgetriangle}. The normal vector $\lambda_{f_3}^\tau$ to the side $f_3$ is smoothly nudged, taking an irrational polytope to a rational one. The grid is the integer lattice $\Z^2$.}
\label{fig 0}
\end{figure}
Let $n=2$ and $Q=\Z^3$, with standard generators $e_1, e_2, e_3$ sent to $\partial_1^\tau,\partial_2^\tau,$ and $\partial_3^\tau$, respectively. We describe a deformation of stacky moment polytopes where $\Delta$ is a right triangle, with sides labeled $f_1,f_2,f_3$. Let $q_{f_i} = e_i$, and fix the normal vectors
\[
\partial_1^\tau = \lambda_{f_1}^\tau = -(1,0)\in \R^2; \qquad \partial_2^\tau = \lambda_{f_2}^\tau = -(0,1)\in \R^2.
\]
Let $L_{f_1}^\tau =L_{f_2}^\tau = 0$, and let $L_{f_3}^\tau>0$ be constant. The third side's normal vector $\partial_3^\tau= \lambda^\tau_{f_3}\in \R^2_{>0}$ will vary smoothly in $\tau$. By picking a smooth path from an irrational $\partial_3^0 \in (\R^2_{>0})\backslash \Q_{>0}^2$ to a rational $\partial_3^1 \in \Q_{>0}^2$ we get a deformation from an irrational triangle to a rational triangle. Note however we have $Q=\Z^3$ throughout.
\end{example}

\begin{definition}
Fix a deformation of stacky tori $\G=\B(\partial^\bu\colon Q\to \R^n)$, with Lie algebra $\g=\Lie(\G^\tau)$. A \emph{deformation of toric stacks} is a tuple
\[
(\X, p_\X \colon \X \to (0,1),\G, \stackmorphism{\pi},  \bmu\colon \X \to \g^*),
\]
where $\X$ is a $(2n+1)$-dimensional \'etale stack, $p_\X\colon \X \to (0,1)$ is a surjective submersion as in \cite[Definition 2.9]{bursztyn-noseda-zhu;principal-stacky-groupoids}, there is a left action $\G\times_{(0,1)} \X \to \X$ of $\G$ on $\X$ as in \cite[Definition 3.15]{bursztyn-noseda-zhu;principal-stacky-groupoids}, and $\stackmorphism{\pi}$ is a bivector field on $\X$. We impose the following requirements. First, for each $\tau\in (0,1)$, the bivector field $\stackmorphism{\pi}$ restricted to the fiber $\X^\tau = p\n_\X(\tau)$ must take values in $\wedge^2T\X^\tau$. Second, for each $\tau\in (0,1)$, there must exist a symplectic form $\bomega^\tau\in \Omega^2(\X^\tau)$ so that $\tilde{\bomega}^\tau = (\tilde{\stackmorphism{\pi}}|_{\X^\tau})\n$, where
\[
\tilde{\stackmorphism{\pi}}|_{\X^\tau}\colon \Omega^1(\X^\tau)\to \Vect(\X^\tau); \quad\tilde{\bomega}^\tau \colon \Vect(\X^\tau) \to \Omega^1(\X^\tau)
\]
are the linear maps induced by contraction with $\stackmorphism{\pi}|_{\X^\tau}$ and $\bomega^\tau$, respectively. Finally, the tuple 
\[
(\X^\tau,\bomega^\tau, \B(\partial^\tau\colon Q\to \R^n), \bmu^\tau = \bmu|_{\X^\tau})
\] must be a compact toric symplectic stack.
\end{definition}

\begin{theorem}
\label{deformations}
Let $\G=\B(\partial^\bu\colon Q\to \R^n)$ be a deformation of stacky tori, with Lie algebra $\g=\Lie(\G^\tau)\cong \R^n$. Let $(\Delta^\tau,\G^\tau, \{\Lambda_{f}^\tau\}_{f\in F}, \{q_f\}_{f\in F^{max}}, \{L_f\}_{f\in F^{\max}})$ be a deformation of stacky moment polytopes. Then, there exists a deformation of toric stacks
\[
(\X, p_\X \colon \X \to (0,1),\G, \stackmorphism{\pi},  \bmu\colon \X \to \g^*)
\]
so that for any $\tau\in(0,1)$, the decorated stacky moment polytope of $(\X^\tau,\bomega^\tau, \G^\tau, \bmu^\tau )$ is $(\Delta^\tau,\G^\tau, \{\Lambda_{f}^\tau\}_{f\in F})$.
\end{theorem}

\begin{proof}
Let $d=|F^{max}|$ be the number of facets of $\Delta^\tau$, and enumerate the facets $f_1,\dots,f_d$. Let
$q_i := q_{f_i},$ and let $\lambda^\tau_i:= \lambda^\tau_{f_i} =\partial^\tau(q_i).$
Consider the Morita morphism of Lie groupoids
\[
\begin{tikzcd}
\lambda^* (Q\ltimes (\R^n\times (0,1))) \arrow[d, shift right] \arrow[d,shift left] \arrow[r]
& Q\ltimes (\R^n\times (0,1)) \arrow[d, shift right] \arrow[d,shift left] \\
\R^d\times (0,1) \arrow[r, "\lambda"]
& \R^n\times (0,1)
\end{tikzcd}
\]
where, for the standard basis vectors $e_1,\dots,e_d$ of $\R^n$, we define
\[
\lambda\left(\sum_i a_i e_i, \tau\right) = \left(\sum_i a_i \lambda^\tau_i, \tau\right).
\]
We will sometimes write $\lambda^\tau(x) = \lambda(x,\tau)$. At a fixed $\tau\in (0,1)$, this map of Lie groupoids restricts to a Morita morphism of Lie 2-groups $\lambda^*(Q\ltimes(\R^n\times \{\tau\}))\to Q\ltimes (\R^n \times \{\tau\})$.

Consider the bundle (over $(0,1)$) of Lie 2-groups, which is the direct product of the pair groupoid $\Z^d \times \Z^d \rightrightarrows \Z^d $ and the trivial groupoid $(0,1)\rightrightarrows (0,1)$. There is a naturally defined embedding of this bundle of Lie $2$-groups into $\lambda^*(Q\ltimes(\R^n\times(0,1)))$. On arrows, the embedding is
\begin{align*}
 \Z^d\times \Z^d\times (0,1) & \hookrightarrow \left(\R^d\times (0,1)\right)\times_{\lambda, \R^n\times (0,1), s} \Big(Q \ltimes(\R^n\times (0,1)) \Big)\times_{t, \R^n\times(0,1),\lambda}\left( \R^d\times (0,1)\right) \\
 \left(\sum_i a_i e_i, \sum_i b_i e_i, \tau\right) &\mapsto \left(\left(\sum_i a_ie_i, \tau\right), \left( \sum_i (b_i-a_i)q_i , \sum_i a_i \lambda_i^\tau ,\tau\right), \left(\sum_i b_i e_i, \tau\right)
\right).
\end{align*}
For a given $\tau$, the image of $\Z^d\times \Z^d\times\{\tau\}\rightrightarrows \Z^d\times\{\tau\}$ is a Lie 2-subgroup of $\lambda^*(Q\ltimes (\R^n\times\{\tau\}))$.
The quotient 
\[
\begin{tikzcd}
 \lambda^* (Q\ltimes (\R^n\times (0,1)))/(\Z^d\times \Z^d\times (0,1)) \arrow[d, shift right] \arrow[d,shift left] \\
 \R^d/ \Z^d \times (0,1)
\end{tikzcd}
\]
is isomorphic to the bundle (over $(0,1)$) of Lie 2-groups $N\times_{(0,1)} ( \R^d/\Z^d \times (0,1))\rightrightarrows \R^d/\Z^d\times (0,1)$. 
Here, 
\[
N=s\n(0,(0,1))\subset \lambda^*  (Q\ltimes (\R^n\times (0,1)))/(\Z^d\times \Z^d\times (0,1))
\]
 is a smooth manifold with a submersion $p_N\colon N\to (0,1)$. Each fiber $N^\tau = p_N\n(\tau)$ is a Lie subgroup of $ \lambda^*(Q\ltimes (\R^n\times \{\tau\})))/(\Z^d\times \Z^d\times \{\tau\})$, and so for each $\tau$ the target map $t$ restricts to a Lie group homomorphism $\alpha^\tau \colon N^\tau \to \R^d/\Z^d$. At a fixed $\tau\in (0,1)$, the subgroupoid of $N\times_{(0,1)} ( \R^d/\Z^d \times (0,1))\rightrightarrows \R^d/\Z^d\times (0,1)$ over $\tau$ has the following Lie 2-group structure:
 \[
\Big( N\times_{(0,1)} ( \R^d/\Z^d \times \{\tau\})\rightrightarrows \R^d/\Z^d\times \{\tau\} \Big) \cong \Big( N^\tau \ltimes R^d/\Z^d \rightrightarrows \R^d/\Z^d \Big).
 \]
 Let $\mathfrak{n}^\tau = \Lie(N^\tau)$ for some $\tau\in (0,1)$. Its isomorphism type is constant in $\tau$. 
 
 The bundle of Lie 2-groups $N\times_{(0,1)} (\R^d/\Z^d\times (0,1))\rightrightarrows \R^d/\Z^d \times (0,1)$ is a groupoid presentation of $\G$. What's more, the map $\alpha^\tau$ is an immersion, and so for a fixed $\tau\in (0,1)$ we have a short exact sequence
\[
\begin{tikzcd}
0 \arrow[r] & \mathfrak{n}^\tau \arrow[r, "\Lie(\alpha^\tau)"] & \Lie(\R^d) \arrow[r,"\Lie(\lambda^\tau)"] & \g\arrow[r] & 0
\end{tikzcd}
\]
For any $\tau$, there is a canonical isomorphism $\Lie(\R^d)/\mathfrak{n}^\tau\cong \g$. Dually, we have the short exact sequence
\[
\begin{tikzcd}
0 \arrow[r] & \g^* \arrow[r, "\Lie(\lambda^\tau)^*"] & \Lie(\R^d)^* \arrow[r,"\Lie(\alpha^\tau)^*"] & (\mathfrak{n}^\tau)^* \arrow[r] & 0
\end{tikzcd}
\]
 and the isomorphism $\ann(\mathfrak{n}^\tau)\cong \g^*$.
 
Consider the Poisson manifold $M=\C^d\times (0,1)$, with coordinates $z_1,\overline{z}_1,\dots, z_d,\overline{z}_d,\tau$, and bivector field
\[
\pi_M = 2i \sum_{j=1}^d \frac{\partial}{\partial z_j} \wedge \frac{\partial}{\partial\overline{z}_j}.
\]
Let $p_M\colon M\to (0,1)$ be the projection to the last coordinate.
There is a Hamiltonian action of $\R^d/\Z^d$ on $M$,
\[
(a_1,\dots, a_d)\cdot (z_1,\dots,z_d,\tau) = (e^{2\pi ia_1} z_1,\dots, e^{2\pi i a_d} z_d, \tau)
\]
with moment map
\[
\mu_M (z_1,\dots,z_d,\tau) = \sum_j (-\pi |z_j|^2+L_j^\tau) e_j^*,
\]
where the vectors $\{e_i^*\}$ are the dual to the standard basis $\{e_i\}\subset \g\cong \R^n$. 

Let 
\[
X = \{ (z,\tau) \in M \mid \Lie(\alpha^\tau)^*\circ \mu_M(z,\tau) = 0 \}.
\]
Then $X$ is a smooth manifold and the projection $p_X = p_M|_X \colon X\to (0,1)$ is a proper submersion. There is an action of $N\to (0,1)$ on $X\to (0,1)$, which is 
\begin{align*}
N\times_{(0,1)} X & \to X\\
(n,\tau)\cdot (z,\tau) & =( \alpha^\tau(n)\cdot z,\tau).
\end{align*}
Let $X^\tau=p_X\n(\tau)$. Consider the groupoid $N \times_{(0,1)} X\rightrightarrows X$, which over a fixed $\tau\in (0,1)$ is the action groupoid $N^\tau\ltimes X^\tau\rightrightarrows X^\tau$. We consider $N\times_{(0,1)} X\rightrightarrows X$ as a bundle of Lie groupoids over $(0,1)$, with structure map $p_X\colon X \to (0,1)$. There is an action of the bundle of Lie 2-groups $N\times_{(0,1)} (\R^d/\Z^d \times(0,1))\rightrightarrows (\R^d/\Z^d \times (0,1))$ on $N\times_{(0,1)} X\rightrightarrows X$. On arrows, the action is
\begin{align*}
(N\ltimes(\R^d/\Z^d \times(0,1)))\times_{(0,1)}(N\times_{(0,1)} X) & \to N\times_{(0,1)} X \\(
(n, g,\tau) , (n',z,\tau)) & \mapsto  (nn', g\cdot z, \tau).
\end{align*}
Then, it is straightforward to check that the Lie groupoid $N\times_{(0,1)} X\rightrightarrows X$, together with the map $p_X$, the action of  $N\times_{(0,1)} (\R^d/\Z^d\times(0,1))\rightrightarrows \R^d\times (0,1)$, the $N$-invariant bivector $\pi_M|_X$, and the moment map $\mu_X= \mu_M|_X$ present a deformation of toric stacks $(\X, p_\X, \G, \stackmorphism{\pi}, \bmu)$. 

For a given $\tau$, the stacky moment polytope of $\X^\tau$ is $(\Delta^\tau,\G^\tau,\{\Lambda^\tau_f\}_{f\in F})$. We will verify that the lattices given by \eqref{polytopelabel} coincide with the $\Lambda_f^\tau$. We leave the rest of the verification (which is similar to the verification in \cite{lerman-tolman;hamiltonian-torus-actions-symplectic-orbifolds} and \cite{delzant;hamiltoniens-periodiques}) to the reader. It is enough to consider the lattices $\Lambda_v^\tau$ for vertices $v\in V$ of $\Delta^\tau$; without loss of generality assume $v=\overline{f_1}\cap \overline{f_2}\cap \cdots \cap \overline{f_n}$. Consider the linear section $\sigma_0\colon \R^n\to \R^d$ of $\lambda^\tau$ which is given on generators by $\lambda_i^\tau \mapsto e_i$, for $i=1,\dots, n$. Then together with
\begin{align*}
\sigma_1\colon Q\ltimes (\R^n \times \{\tau\}) & \to \left(\R^d\times \{\tau\}\right)\times_{\lambda, \R^n\times (0,1), s} \Big(Q \ltimes(\R^n\times \{\tau\}) \Big)\times_{t, \R^n\times(0,1),\lambda}\left( \R^d\times \{\tau\}\right) \\
(q,x,\tau) & \mapsto \Big((\sigma_0(x),\tau),(q,x,\tau),(\sigma_0(\partial^\tau(q)+x),\tau)\Big),
\end{align*}
the Lie 2-group morphism $\sigma_\bu\colon Q\ltimes (\R^n \times \{\tau\}) \to (\lambda^\tau)^*(Q\ltimes (\R^n\times \{\tau\}))$ is a Morita morphism of Lie 2-groups. Since the action of $(\lambda^\tau)^*(Q\ltimes (\R^n\times \{\tau\}))$ on $N^\tau\ltimes X^\tau\rightrightarrows X^\tau$ is Hamiltonian, by applying $\sigma_\bu$ to $Q\ltimes \R^n = Q\ltimes (\R^n\times\{\tau\})$, one arrives at a Hamiltonian action of $Q\ltimes \R^n\rightrightarrows \R^n$ on $N^\tau\ltimes X^\tau \rightrightarrows X^\tau$. This presents the action of $\G^\tau$ on $\X^\tau$.

Let $x\in \mu_X\n(v)\cap X^\tau$, then
\[
\sigma_0(\R^n) = \bigoplus_{i=1,\dots, n} \R e_i \subset \stab_{\R^d} (x).
\]
The point $x\in X^\tau$ is then stabilized under this action of $\R^n$. So we may restrict the action of $Q\ltimes \R^n \rightrightarrows \R^n$  on $N^\tau\ltimes X^\tau \rightrightarrows X^\tau$ to the subgroupoid $(N^\tau\ltimes X^\tau)|_{x} = (\stab_{N_\tau}(x) \ltimes \{x\} \rightrightarrows \{x\})$. This is a presentation of the action of  $\G^\tau$ on $\X^\tau|_x$.
Then, by the description of $\Lambda^\tau_v$ in the proof of Proposition \ref{polyprop}, 
\begin{align*}
\Lambda_v^\tau & = \partial^\tau \Big(\ker(Q \to \stab_{N^\tau}(x))\Big) \\
& = \partial^\tau \Big(\bigoplus_{i=1,\dots, n} \Z q_i \Big) \\
& = \bigoplus_{i=1,\dots, n} \Z \lambda_i^\tau.
\end{align*}
From this it follows the subgroup of $\partial(Q)$ attached to any facet $f_i$ is indeed $\Lambda^\tau_i$.
\end{proof}

Recall from Remark \ref{toricorbifold} that a compact toric symplectic stack $(\X,\bomega,\G, \bmu)$ is a compact toric symplectic orbifold if $\G\cong \R^n/\Z^n$ and $\X$ is an effective proper stack. An \emph{ineffective compact toric symplectic orbifold} is a compact toric symplectic stack of the form
\[
(\B(R\rightrightarrows \star) \times \X, \bomega, \B(R\rightrightarrows \star) \times \G, \bmu),
\]
where $(\X,\bomega,\G, \bmu)$ is a compact toric symplectic orbifold and $R$ is a finitely generated abelian group.  Combining Proposition \ref{deformprop} and Theorem \ref{deformations} gives the following. 

\begin{corollary}
Any compact toric symplectic stack admits a deformation to an ineffective compact toric symplectic orbifold.
\end{corollary}

\bibliographystyle{amsplain}

\bibliography{hamilton}

\def\cprime{$'$}
\providecommand{\bysame}{\leavevmode\hbox to3em{\hrulefill}\thinspace}
\providecommand{\MR}{\relax\ifhmode\unskip\space\fi MR }
\providecommand{\MRhref}[2]{%
  \href{http://www.ams.org/mathscinet-getitem?mr=#1}{#2}
}
\providecommand{\href}[2]{#2}
\begin{thebibliography}{10}

\bibitem{battaglia-prato-zaffran;one-parameter}
F.~Battaglia, E.~Prato, and D.~Zaffran, \emph{Hirzebruch surfaces in a
  one\textendash{}parameter family}, Bollettino dell'Unione Matematica Italiana
  \textbf{12} (2019), 293--305.

\bibitem{behrend-xu;stacks-gerbes}
K.~Behrend and P.~Xu, \emph{Differentiable stacks and gerbes}, J. Symplectic
  Geom. \textbf{9} (2011), no.~3, 285--341.

\bibitem{berwick-evans-lerman;lie-2-algebras-vector}
D.~Berwick-Evans and E.~Lerman, \emph{Lie $2$-algebras of vector fields},
  eprint, 2016, \path{arXiv:1609.03944}.

\bibitem{bursztyn-noseda-zhu;principal-stacky-groupoids}
H.~Bursztyn, F.~Noseda, and C.~Zhu, \emph{Principal actions of stacky {L}ie
  groupoids}, eprint, 2015, \path{arXiv:1510.09208}.

\bibitem{crainicDeformationsLieGroupoids2015}
M.~Crainic, J.~Mestre, and I.~Struchiner, \emph{Deformations of {{Lie}}
  groupoids}, arXiv:1510.02530 (2015).

\bibitem{crainic-moerdijk;foliation-cyclic}
M.~Crainic and I.~Moerdijk, \emph{Foliation groupoids and their cyclic
  homology}, Adv. Math. \textbf{157} (2001), no.~2, 177--197.

\bibitem{delzant;hamiltoniens-periodiques}
T.~Delzant, \emph{Hamiltoniens p\'eriodiques et images convexes de
  l'application moment}, Bull. Soc. Math. France \textbf{116} (1988), no.~3,
  315--339.

\bibitem{geraschenko-satriano;toric}
A.~Geraschenko and M.~Satriano, \emph{Toric stacks {I}: {T}he theory of stacky
  fans}, Trans. Amer. Math. Soc. \textbf{367} (2015), no.~2, 1033--1071.

\bibitem{geraschenko-satriano;toric2}
\bysame, \emph{Toric stacks {II}: {I}ntrinsic characterization of toric
  stacks}, Trans. Amer. Math. Soc. \textbf{367} (2015), no.~2, 1073--1094.

\bibitem{giraud;cohomology-nonabeliene}
J.~Giraud, \emph{Cohomologie non abelienne}, Grundlehren der mathematischen
  Wissenschaften, vol. 179, Springer-Verlag, Berlin Heidelberg, 1971.

\bibitem{hepworth;vector-flow-stack}
R.~Hepworth, \emph{Vector fields and flows on differentiable stacks}, Theory
  Appl. Categ. \textbf{22} (2009), 542--587.

\bibitem{hilgert-neeb-plank;symplectic-convexity;compositio}
J.~Hilgert, K.-H. Neeb, and W.~Plank, \emph{Symplectic convexity theorems and
  coadjoint orbits}, Compositio Math. \textbf{94} (1994), no.~2, 129--180.

\bibitem{hoffman-sjamaar;hamiltonian-stack}
B.~Hoffman and R.~Sjamaar, \emph{Stacky hamiltonian actions and symplectic
  reduction}, eprint, 2018, with an appendix by C. Zhu,
  \path{arXiv:1808.01003}. To appear in International Mathematics Research
  Notices.

\bibitem{katzarkovDefinitionNoncommutativeToric2014}
L.~Katzarkov, E.~Lupercio, L.~Meersseman, and A.~Verjovsky, \emph{The
  definition of a non-commutative toric variety}, Contemporary {{Mathematics}}
  (U.~Tillmann, S.~Galatius, and D.~Sinha, eds.), vol. 620, {American
  Mathematical Society}, {Providence, Rhode Island}, 2014, pp.~223--250 (en).

\bibitem{lerman;orbifolds-as-stacks}
E.~Lerman, \emph{Orbifolds as stacks?}, Enseign. Math. (2) \textbf{56} (2010),
  no.~3-4, 315--363.

\bibitem{lerman-malkin;deligne-mumford}
E.~Lerman and A.~Malkin, \emph{Hamiltonian group actions on symplectic
  {D}eligne-{M}umford stacks and toric orbifolds}, Adv. Math. \textbf{229}
  (2012), no.~2, 984--1000.

\bibitem{lerman-tolman;hamiltonian-torus-actions-symplectic-orbifolds}
E.~Lerman and S.~Tolman, \emph{Hamiltonian torus actions on symplectic
  orbifolds and toric varieties}, Trans. Amer. Math. Soc. \textbf{349} (1997),
  no.~10, 4201--4230.

\bibitem{lin-sjamaar;presymplectic}
Y.~Lin and R.~Sjamaar, \emph{Convexity properties of presymplectic moment
  maps}, J. Symplectic Geom., accepted for publication,
  \path{arXiv:1706.00520}.

\bibitem{metzler;topological-smooth-stack}
D.~Metzler, \emph{Topological and smooth stacks}, eprint, 2003,
  \path{arXiv:math/0306176}.

\bibitem{moerdijk-mrcun;foliations-groupoids}
I.~Moerdijk and J.~Mr\v{c}un, \emph{Introduction to foliations and {L}ie
  groupoids}, Cambridge Studies in Advanced Mathematics, vol.~91, Cambridge
  University Press, Cambridge, 2003.

\bibitem{moerdijkOrbifoldsSheavesGroupoids1997}
I.~Moerdijk and D.~Pronk, \emph{Orbifolds , {{Sheaves}} and {{Groupoids
  Dedicated}} to the memory of}, K-Theory \textbf{12} (1997), 3--21.

\bibitem{noohi;foundations-topological-stacks}
B.~Noohi, \emph{Foundations of topological stacks}, eprint, 2005,
  \path{arXiv:math/0503247}.

\bibitem{prato;non-rational-symplectic}
E.~Prato, \emph{Simple non-rational convex polytopes via symplectic geometry},
  Topology \textbf{40} (2001), no.~5, 961--975.

\bibitem{pronkEtenduesStacksBicategories1996}
D.~Pronk, \emph{Etendues and stacks as bicategories of fractions}, Compositio
  Mathematica \textbf{102} (1996), no.~3, 243--303 (en).

\bibitem{ratiu-zung;presymplectic}
T.~Ratiu and N.~T. Zung, \emph{Presymplectic convexity and (ir)rational
  polytopes}, eprint, 2017, \path{arXiv:1705.11110}.

\bibitem{thurston;alexander}
W.~Thurston~(\path{https://mathoverflow.net/users/9062/bill-thurston}),
  \emph{Complement of a totally disconnected set in the plane}, Math{O}verflow
  post, 2011,
  \path{https://mathoverflow.net/questions/55718/complement-of-a-totally-disconnected-closed-set-in-the-plane}.

\bibitem{trentinaglia-zhu;strictification}
G.~Trentinaglia and C.~Zhu, \emph{Strictification of \'etale stacky {L}ie
  groups}, Compos. Math. \textbf{148} (2012), no.~3, 807--834.

\end{thebibliography}

\end{document}